\documentclass[11pt]{amsart}
\usepackage{amsmath, amsthm, mathabx,amscd, amsfonts, amssymb, hyperref, mathrsfs, textcomp, epsfig, a4wide, graphicx, IEEEtrantools, tikz, verbatim, xypic}

\usetikzlibrary{matrix}

\newtheorem{thm}{Theorem}
\newtheorem{prop}{Proposition}
\newtheorem{lemma}{Lemma}
\newtheorem{cor}{Corollary}
\newcommand{\spin}{\mathfrak{s}}
\theoremstyle{definition}
\newtheorem{defn}{Definition}

\theoremstyle{remark}
\newtheorem{remark}{Remark}
\newtheorem{example}{Example}

     \RequirePackage{rotating}                   
    \def\HSt{%
       \setbox0=\hbox{$\widehat{\mathit{HS}}$}
       \setbox1=\hbox{$\mathit{HS}$}
       \dimen0=1.1\ht0
       \advance\dimen0 by 1.17\ht1
       \smash{\mskip2mu\raise\dimen0\rlap{%
          \begin{turn}{180}
              {$\widehat{\phantom{\mathit{HS}}}$}
           \end{turn}} \mskip-2mu    
                \mathit{HS}
    }{\vphantom{\widehat{\mathit{HS}}}}{}}

     \RequirePackage{rotating}                   
    \def\HMt{%
       \setbox0=\hbox{$\widehat{\mathit{HM}}$}
       \setbox1=\hbox{$\mathit{HM}$}
       \dimen0=1.1\ht0
       \advance\dimen0 by 1.17\ht1
       \smash{\mskip2mu\raise\dimen0\rlap{%
          \begin{turn}{180}
              {$\widehat{\phantom{\mathit{HM}}}$}
           \end{turn}} \mskip-2mu    
                \mathit{HM}
    }{\vphantom{\widehat{\mathit{HM}}}}{}}

    \newcommand{\HSb}{\overline{\mathit{HS}}}
\newcommand{\HSf}{\widehat{\mathit{HS}}}

\newcommand{\CPbar}{\overline{\mathbb{C}P}^2}
\newcommand{\Pin}{\mathrm{Pin}(2)}
\newcommand{\Cr}{\mathfrak{C}}
\newcommand{\Rin}{\mathcal{R}}
\newcommand{\ztwo}{\mathbb{F}}

\newcommand{\T}{\mathcal{T}}
\newcommand{\V}{\mathcal{V}}
\begin{document}

\title{The surgery exact triangle in $\text{PIN}(2)$-monopole Floer homology} 

\author{Francesco Lin}
\address{Department of Mathematics, Massachusetts Institute of Technology} 
\email{linf@math.mit.edu}

\begin{abstract}
We prove the existence of an exact triangle for the $\Pin$-monopole Floer homology groups of three manifolds related by specific Dehn surgeries on a given knot. Unlike the counterpart in usual monopole Floer homology, only two of the three maps are those induced by the corresponding elementary cobordism. We use this triangle to describe the Manolescu correction terms of the manifolds obtained by $(\pm1)$-surgery on alternating knots.
\end{abstract}
\maketitle

\section*{Introduction}
The goal of this paper is to describe the relation between the $\Pin$-monopole Floer homology groups of three manifolds which are obtained from a given one by Dehn surgery on a knot. $\Pin$-monopole Floer homology is a gauge-theoretic invariant of closed connected and oriented three manifolds. It was introduced by the author in \cite{Lin} as the analogue of Manolescu's $\Pin$-equivariant Seiberg-Witten Floer homology groups for rational homology spheres (\cite{Man2}) in the context of Kronheimer and Mrowka's monopole Floer homology (\cite{KM}). In particular it can be used to give an alternative disproof of the long standing Triangulation Conjecture, see also \cite{Man3} for a nice survey.
\par
Unlike Manolescu's construction, the definition in \cite{Lin} works for every closed oriented connected three manifold $Y$. It associates in a functorial way to each such $Y$ three groups fitting in a long exact sequence
\begin{equation}\label{longex}
\dots\stackrel{i_*}{\longrightarrow}\HSt_{\bullet}(Y)\stackrel{j_*}{\longrightarrow} \HSf_{\bullet}(Y)\stackrel{p_*}{\longrightarrow} \HSb_{\bullet}(Y)\stackrel{i_*}{\longrightarrow}\dots
\end{equation}
which are read respectively \textit{H-S-to}, \textit{H-S-from} and \textit{H-S-bar}. These are also (relatively) graded topological modules over the ring 
\begin{equation*}
\Rin=\ztwo[[V]][Q]/(Q^3)
\end{equation*}
where $V$ and $Q$ have degree respectively $-4$ and $-1$ and $\ztwo$ is the field with two elements. The ring $\Rin$ should be thought as the completion (with reverse gradings) of the cohomology of the classifying space of 
\begin{equation*}
\Pin= S^1\times jS^1\subset\mathbb{H}.
\end{equation*}
These objects are obtained by studying the negative gradient flow of the Chern-Simons-Dirac functional on the three manifold from a Floer-theoretic point of view. For a general spin$^c$ structure $\spin$, the equations have an $S^1$ symmetry, but when $\spin$ is induced by a genuine spin structure quaternionic geometry comes into play and the equations acquire a $\Pin$ symmetry. $\Pin$-monopole Floer homology is then constructed by suitably exploiting this extra input.
\\
\par
In the present paper we develop in this setting one of the essential features of Floer homology theories for three manifolds, namely \textit{surgery exact triangles}. Their construction dates back to Floer's original instanton invariants \cite{Flo}, and they turn out to be a key tool for topological applications of Floer theories to three manifold topology. For example, the version for monopole Floer homology is introduced in \cite{KMOS} and is used to prove a conjecture of Gordon (\cite{Gor}) regarding a surgery characterization of the unknot.
\par
Suppose we are given a connected compact oriented three manifold $Z$ with torus boundary $\partial Z$, and let $\gamma_i$, $i=1,2,3$ be oriented simple closed curves having intersection numbers
\begin{equation*}
\gamma_1\cdot \gamma_2=\gamma_2\cdot\gamma_3=\gamma_3\cdot\gamma_1=-1.
\end{equation*}
Call $Y_i$ the three manifold obtained by Dehn filling $\partial Z$ along $\gamma_i$. Associated to this data there is a canonical cobordism $W_i$ from $Y_i$ to $Y_{i+1}$ given by a single $2$-handle attachment along a suitably framed copy of the knot. The key observation for our purposes is that among these three cobordisms exactly two are spin, while the third is not. The typical examples of such triples are given by
\begin{equation*}
\infty, p,p+1\quad\text{and}\quad
0, 1/(q+1),1/q
\end{equation*}
surgeries on a knot in the three sphere, where both $p$ and $q$ are integers. In the first case the non spin cobordism is $W_1$ if $p$ is odd and $W_3$ if $p$ is even. In the second case the non spin cobordism is always $W_2$. The following is then main result of the paper.
\begin{thm}\label{exacttr}
Suppose that the non spin cobordism is $W_3$. There exists a map
\begin{equation*}
\check{F}_3: \HSt_{\bullet}(Y_3)\rightarrow \HSt_{\bullet}(Y_1)
\end{equation*}
of $\Rin$-modules such that the triangle
\begin{center}
\begin{tikzpicture}
\matrix (m) [matrix of math nodes,row sep=2em,column sep=1.5em,minimum width=2em]
  {
  \HSt_{\bullet}(Y_2) && \HSt_{\bullet}(Y_3)\\
  &\HSt_{\bullet}(Y_1) &\\};
  \path[-stealth]
  (m-1-1) edge node [above]{$\HSt_{\bullet}(W_2)$} (m-1-3)
  (m-2-2) edge node [left]{$\HSt_{\bullet}(W_1)$} (m-1-1)
  (m-1-3) edge node [right]{$\check{F}_3$} (m-2-2)  
  ;
\end{tikzpicture}
\end{center}
is exact. The map $\check{F}_3$ is uniquely defined for each pair of three manifolds $Y_3$, $Y_1$ such that the latter is obtained by Dehn surgery and the corresponding elementary cobordism given by a $2$-handle attachment is not spin. The same statement holds for the from and bar versions.
\end{thm}

The map $\check{F}_3$ in the statement above is genuinely different from the one induced by the cobordism $W_3$ as defined in \cite{Lin}. In fact already in simple examples (see Section \ref{computations}) the triangle where all the maps are the ones induced by cobordisms the composite maps are not necessarily zero. This is due to the fact that in $\Pin$-monopole Floer homology there are interesting modulo four periodicity phenomena to take account of (see for example the blow-up formula in Section \ref{blowups}). Unfortunately, we cannot provide a more geometric interpretation of this map at the moment.

\vspace{0.5cm}

Before discussing examples of the surgery exact triangle, we compute the $\Pin$-monopole Floer homology groups of the homology spheres obtained by surgery on the trefoil knot. As these are (up to orientation reversal) the Brieskorn spheres $\Sigma(2,3,6n\pm1)$, we recover in our setting the results in \cite{Man2}. In the statement we adopt the following notation. For any rational number $d$ let $\V^+_d$ be the $\ztwo [[V]]$-module $\ztwo[V^{-1},V]]/V\ztwo [[V]]$, where the grading is shifted so that the element $1$ has degree $d$. The multiplication by $V$ has degree $-4$, and the double square brackets indicate that we are considering a quotient of the ring of Laurent power series. We will refer as these $\ztwo [[V]]$-modules simply as \textit{towers}, and they will arise as the image of the map $i_*$ in the long exact sequence (\ref{longex}). We also denote a trivial $\Rin$-summand of the form $\ztwo^k$ all concentrated in degree $d$ by $\ztwo^k\langle d\rangle$.

\begin{thm}\label{trefoil}
We have for $k\geq0$ the isomorphisms of graded $\Rin$-modules:
\begin{align*}
\HSt_{\bullet}(\Sigma(2,3, 12k+5))&= \V^+_4\oplus \V^+_3\oplus\V^+_2\oplus \ztwo^k\langle1\rangle\\
\HSt_{\bullet}(\Sigma(2,3, 12k+1))&=\V^+_2\oplus \V^+_1\oplus\V^+_0\oplus\ztwo^k\langle-1\rangle
\end{align*}
where the action of $Q$ (which has degree $-1$) is an isomorphism from the first tower to the second second tower, an isomorphism from the second tower to the third tower, and zero otherwise. The direct sum of the three towers is the image of $i_*$. Similarly, for $k>0$ we have:
\begin{align*}
\HSt_{\bullet}(\Sigma(2,3, 12k-1))&= \V^+_2\oplus\V^+_1\oplus \V^+_4\oplus \ztwo^{k-1}\langle1\rangle\\
\HSt_{\bullet}(\Sigma(2,3, 12k-5))&= \V^+_0\oplus\V^+_{-1}\oplus\V^+_2\oplus \ztwo^{k-1}\langle-1\rangle
\end{align*}
where the action of $Q$ (which has degree $-1$) is an isomorphism from the first tower to the second second tower, maps the second tower onto the third tower, and is zero otherwise. Again the direct sum of the three towers is the image of $i_*$.
\end{thm}
The idea behind the computation is that if the usual monopole Floer homology $\HMt_{\bullet}$ is simple enough, the $\Pin$ counterpart can be determined in a purely algebraic way from the Gysin exact sequence
\begin{equation}
\dots\stackrel{\cdot Q}{\longrightarrow} \HSt_k(Y)\stackrel{\iota_*}{\longrightarrow} \HMt_k(Y)\stackrel{\pi_*}{\longrightarrow} \HSt_k(Y)\stackrel{\cdot Q}{\longrightarrow}\HSt_{k-1}(Y)\stackrel{\iota_*}{\longrightarrow}\dots
\end{equation} 
introduced in Section $4.3$ of \cite{Lin}. This granted, we will discuss how the Floer groups of this manifolds fit in the surgery exact triangle, and this will provide a model for more interesting computations. Similarly, we compute the $\Pin$-monopole Floer homology groups of the homology spheres obtained by surgery on the figure eight knot.

\begin{thm}\label{figure8}
Denote by $E_n$ the manifold obtained by $1/n$ surgery on the figure eight knot. Let $\spin_0$ be the only self-conjugate spin$^c$ structure on $E_0$. Then we have the isomorphism of graded $\Rin$-modules:
\begin{equation*}
\HSt_{\bullet}(E_0,\spin_0)\cong \V^+_1\oplus \V^+_0\oplus \V^+_{-1}\oplus \V^+_2
\end{equation*}
where the action of $Q$ (which has degree $-1$) is an isomorphism from the first tower to the second, maps the third tower onto the forth tower, and zero otherwise. The group for the other spin$^c$ structures vanishes, and the map $i_*$ is surjective.
Furthermore we have for $k\geq 0$
\begin{equation*}
\HSt_{\bullet}(E_{2k+1})\cong \V^+_{0}\oplus \V^+_{-1}\oplus\V^+_{2}\oplus \ztwo^k\langle-1\rangle
\end{equation*}
and similarly for $k>0$
\begin{equation*}
\HSt_{\bullet}(E_{2k})\cong \V^+_2\oplus \V^+_1\oplus\V^+_0\oplus \ztwo^k\langle-1\rangle.
\end{equation*}
Here the action of $Q$ is the same as the analogous modules appearing in Theorem \ref{trefoil}, and the image of $i_*$ consists exactly of the direct sum of the three towers.
\end{thm}
\vspace{0.5cm}

Recall more in general that for any rational homology sphere $Y$ equipped with a self-conjugate spin$^c$ structure $\spin$ the group $\HSt_{\bullet}(Y,\spin)$, considered as an $\ztwo[[V]]$-module, decomposes as a finite part and the sum of three towers
\begin{equation*}
\V^+_c\oplus \V^+_b\oplus \V^+_a.
\end{equation*}
The action of $Q$ sends the first tower onto the second and the second tower onto the third, and the union of sum of the three towers is the image of the map $i_*$. Manolescu's correction terms are then defined to be the rational numbers
\begin{equation*}
\alpha(Y)\geq\beta(Y)\geq\gamma(Y),
\end{equation*}
all of which have the same fractional part, such that
\begin{equation*}
a=2\alpha(Y),\quad b= 2\beta(Y)+1, \quad c=2\gamma(Y)+2.
\end{equation*}
The inequalities between these quantities follow from the module structure. These are rational lifts of the Rokhlin invariant of $(Y,\spin)$ which are invariant under spin homology cobordisms, and they are integers in the case of a genuine homology sphere. The direct sum of the three towers is an $\Rin$-submodule whose abstract isomorphism class as an absolutely graded $\Rin$-module will be denoted by $\mathcal{S}^+_{\alpha,\beta,\gamma}$. We call this a \textit{standard $\Pin$-module}. For example, the four direct sums of towers appearing in the statement of Theorem \ref{trefoil} are respectively
\begin{equation*}
\mathcal{S}^+_{1,1,1},\quad \mathcal{S}^+_{0,0,0}\quad \mathcal{S}^+_{2,0,0},\quad \mathcal{S}^+_{1,-1,-1}.
\end{equation*}
We can use the surgery exact triangle to determine the Manolescu's correction terms of integral homology spheres obtained by $(\pm1)$-surgery on alternating knots.
\begin{thm}\label{altcorrection}
Let $K$ be an alternating knot with signature $\sigma=\sigma(K)\leq 0$. If $\mathrm{Arf}(K)=0$ then $\alpha,\beta$ and $\gamma$ of the three manifold obtained by $(-1)$-surgery on $K$ are all zero. In the other cases, their value is determined in Table \ref{arf0} and \ref{arf1}.
\end{thm} 
\begin{equation}\label{arf0}
\begin{tabular}{|c|c|}
\hline
$\sigma$ & $\mathrm{Arf}=0$, $n=1$\\
\hline
$-8k$ & $-2k,-2k,-2k$ \\
\hline
$-8k-2$ & $-2k,-2k,-2k-2$ \\
\hline
$-8k-4$ & $-2k,-2k-2,-2k-2$ \\
\hline
$-8k-6$ &$-2k-2,-2k-2,-2k-2$\\
\hline
\end{tabular}
\end{equation}
\begin{equation}\label{arf1}
\begin{tabular}{|c|c|c|}
\hline
$\sigma$  &  $\mathrm{Arf}=1$, $n=1$ &   $\mathrm{Arf}=1$,$n=-1$\\
\hline
$-8k$ &  $-2k+1, -2k-1,-2k-1$  & $1, -2k+1, -2k-1$\\
\hline
$-8k-2$ & $-2k-1,-2k-1,-2k-1$  & $1,-2k-1, -2k-1$\\
\hline
$-8k-4$ & $-2k-1,-2k-1,-2k-1$  & $1,-2k-1,-2k-1$\\
\hline
$-8k-6$ & $-2k-1,-2k-1,-2k-3$  & $1,-2k-1,-2k-1$\\
\hline
\end{tabular}
\vspace{0.5cm}
\end{equation}
The main idea is again that the Floer homology of surgeries on alternating knots is simple enough so that the $\Pin$-case can be recovered from the usual one by means of the Gysin exact sequence. Our computation relies on the knowledge of the monopole Floer homology of these spaces. This follows from the results in the context Heegaard Floer homology provided by \cite{OSalt} and the isomorphism between the two theories (due to Kutluhan-Lee-Taubes \cite{HFHM1}, \cite{HFHM2},\cite{HFHM3}, \cite{HFHM4}, \cite{HFHM5} and Colin-Ghiggini-Honda \cite{CGH}, \cite{CGH1}, \cite{CGH2}, \cite{CGH3}).
\\
\par
A nice consequence of this computation is the existence of homology spheres not homology cobordant to any Seifert fibered space. In \cite{Sto} an example consisting of the connected sum of two Seifert fibered spaces is provided. We use the following analogous obstruction.
\begin{prop}\label{sfs}
For a Seifert fibered rational homology sphere $Y$ equipped with a spin structure $\spin$ either $\alpha(Y,\spin)=\beta(Y,\spin)$ or $\beta(Y,\spin)=\gamma(Y,\spin)$.
\end{prop}
As a consequence of Theorem \ref{altcorrection}, we obtain the following.
\begin{cor}\label{notSFS}
For $k\geq1$ the manifold obtained by $(-1)$-surgery on an alternating knot with signature $-8k$ and Arf invariant $1$ is not homology cobordant to any Seifert fibered space.
\end{cor}

\vspace{0.5cm}

Along the way we will discuss maps on spin cobordisms with $b^+_2=1,2$. We record the following result as we expect it to have interesting topological applications.
\begin{thm}\label{108}
Let $W$ be a smooth spin cobordism between two spin rational homology spheres $(Y_0,\spin_0)$ and $(Y_1,\spin_1)$. If $b_2^+(W)=1$ then the inequalities
\begin{align*}
\alpha(Y_1,\spin_1)&\geq \beta(Y_0,\spin_0)+\frac{1}{8} (b_2^-(W)-1)\\
\beta(Y_1,\spin_1)&\geq \gamma(Y_0,\spin_0)+\frac{1}{8} (b_2^-(W)-1).
\end{align*}
hold. If $b_2^+(W)=2$ then the inequality
\begin{equation*}
\alpha(Y_1,\spin_1)\geq \gamma(Y_0,\spin_0)+\frac{1}{8} (b_2^-(W)-2).
\end{equation*}
holds.
\end{thm}
This should be thought as a generalization of Donaldson's Theorems B and C (see \cite{DK}) regarding closed spin four manifolds with $b_2^+=1,2$, in the same way as Fr\o yshov's result is a generalization of Donaldson's Theorem A. It is analogous in spirit to Kronheimer's Seiberg-Witten theoretic proof that inspired Furuta's work on the $11/8$-conjecture \cite{Fur}.

\par
\vspace{0.5cm}
\textbf{Acknowledgements.} The author would like to thank his advisor Tom Mrowka for the support throughout the development of this project and Tye Lidman and Ciprian Manolescu for the useful comments. Part of the work was carried over while the author was visiting the University of Pisa. He would like to express his gratitude to them, and in particular Bruno Martelli, for their hospitality. This work was partially supported by NSF grants DMS-0805841 and DMS-1005288.

\vspace{1cm}
\section{Overview of $\text{Pin}(2)$-monopole Floer homology}\label{review}
In this section we provide an overview of the the formal properties of $\Pin$-monopole Floer homology and discuss the key steps of its construction as in Chapter $4$ of \cite{Lin}. For a more detailed introduction to the usual case see the first three chapters of \cite{KM}. Throughout the paper we will always work with coefficients in the field with two elements $\ztwo$.

\vspace{0.7cm}

\textbf{The formal structure.} Let $Y$ be a closed connected oriented three manifold. There is a natural action $\jmath$ of the set of spin$^c$ structures $\mathrm{Spin}^c$ given by complex conjugation. We denote the orbits by $[\spin]$, and call the fixed points of this action \textit{self-conjugate spin$^c$-structures}. The first Chern class of such a spin$^c$ structure $\spin$ is always two-torsion. To each self-conjugate spin$^c$ structure $[\spin]$ we associate the (completed) $\Pin$-monopole Floer homology groups
\begin{equation*}
\HSt_{\bullet}(Y,\spin),\qquad \HSf_{\bullet}(Y,\spin),\qquad \HSb_{\bullet}(Y,\spin).
\end{equation*}
There are analogous cohomological versions. These groups carry a relative $\mathbb{Z}$ grading and an absolute $\mathbb{Q}$ grading. They also carry a structure of topological graded module over the ring
\begin{equation*}
\Rin= \ztwo[[V]][Q]/(Q^3)
\end{equation*}
where the actions of $V$ and $Q$ have degree respectively $-4$ and $-1$. These groups should be thought as computing the middle dimensional homology of an infinite dimensional manifold with boundary $\mathcal{B}^{\sigma}(Y,\spin)/\jmath$. In particular, they compute the homology of the space, the homology relative to the boundary and the homology of the boundary. Indeed they fit in the expected long exact sequence (\ref{longex}), and they satisfy a version Poincar\'e duality with respect to orientation reversal.
\par
In the case $\spin$ is not self-conjugate, we define $\HSt_{\bullet}(Y,[\spin])$ to be the usual monopole Floer homology groups
\begin{equation*}
\HMt_{\bullet}(Y,\spin)\equiv \HMt_{\bullet}(Y,\bar{\spin})
\end{equation*}
which are canonically isomorphic. This can be thought as a $\Rin$ module via the coefficient extension
\begin{equation*}
\Rin\rightarrow \ztwo[[U]]
\end{equation*}
obtained by sending $V$ to $U^2$ and $Q$ to $0$. It is convenient to work with the direct sum of all these groups (all but finitely many of which are trivial), and define
\begin{equation*}
\HSt_{\bullet}(Y)=\bigoplus_{[\spin]\in \mathrm{Spin}^c(Y)/\jmath} \HSt_{\bullet}(Y,[\spin]),
\end{equation*}
and similarly for the other versions. This total group does not carry a relative $\mathbb{Z}$ grading anymore, but has a canonical $\mathbb{Z}/2\mathbb{Z}$-grading.
\\
\par
The basic computation is the case of the three sphere. Recall that we have discussed in the introduction the basic modules $\V^+_d$ and $\mathcal{S}^+_{\alpha,\beta,\gamma}$. We then have
\begin{equation*}
\HSt_{\bullet}(S^3)=\mathcal{S}^+_{0,0,0}=\V^+_2\oplus\V^+_1\oplus \V^+_0
\end{equation*}
where in the second description the action of $Q$ (which has degree $-1$) is an isomorphism from the first tower to the second and from the second tower to the third. Also for a rational number $d$ the module $\V_d$ is defined to be the ring of Laurent power series $\ztwo[V^{-1},V]]$ seen as a $\ztwo[[V]]$-module with the grading shifted so that the element $1$ has degree $d$. We then have the isomorphism of $\Rin$-modules
\begin{equation}\label{S}
\HSb_{\bullet}(S^3)=\V_2\oplus \V_1\oplus \V_0.
\end{equation}
The action of $Q$ is an isomorphism from the first summand to the second and from the second summand to the third. We will denote this absolutely graded $\Rin$-module by $\mathcal{S}$.
The map
\begin{equation*}
i_*:\HSb_{\bullet}(S^3)\rightarrow \HSt_{\bullet}(S^3)
\end{equation*}
is surjective. Finally $\HSf_{\bullet}(S^3)$ can be identified with $\Rin\langle-1\rangle$ where again the braces indicate the grading shift.
\\
\par
These groups also satisfy functoriality properties. If $W$ is a connected cobordism from $Y_0$ to $Y_1$, we obtain a homomorphism of $\Rin$-modules
\begin{equation*}
\HSt_{\bullet}(W): \HSt_{\bullet}(Y_0)\rightarrow \HSt_{\bullet}(Y_1),
\end{equation*}
which also decomposes according to the pairs of conjugate spin$^c$ structures on $W$. Furthermore, if $W_0$ is a cobordism from $Y_0$ to $Y_1$ and $W_1$ is a cobordism from $Y_1$ to $Y_2$, we have the composition law
\begin{equation*}
\HSt_{\bullet}(W_1\circ W_0)=\HSt_{\bullet}(W_1)\circ\HSt_{\bullet}(W_0).
\end{equation*}
A more general functoriality property (following from the work of \cite{Blo}) is the following. Suppose $W$ is a cobordism with several incoming ends $Y_-$ and $Y_1,\dots, Y_n$ and one outgoing end $Y_+$. Then there is an induced map
\begin{equation*}
\HSt_{\bullet}(Y_-)\otimes \HSf_{\bullet}(Y_0)\otimes\cdots\otimes \HSf_{\bullet}(Y_n)\rightarrow \HSt_{\bullet}(Y_+).
\end{equation*}
This holds also when the two \textit{to} groups are replaced by their \textit{from} and \textit{bar} counterparts. As a special case of this general construction, the cobordism $I\times Y$ with a ball removed induces a map
\begin{equation*}
\HSt_{\bullet}(Y)\otimes \Rin\langle -1\rangle\rightarrow \HSt_{\bullet}(Y)
\end{equation*}
where we think of the additional boundary component as an incoming $S^3$ end, and we use this map as the definition of the $\Rin$-module structure. 
\\
\par
As briefly discussed in the Introduction, in the case $Y$ rational homology sphere equipped with a self conjugate spin$^c$ structure $\spin$ the groups have a rather simple structure. In particular we have the isomorphism of \textit{relatively} graded $\Rin$-modules
\begin{equation*}
\HSb_{\bullet}(Y,\spin)\cong\HSb_{\bullet}(S^3)\cong \mathcal{S}
\end{equation*}
Analogously the group $\HSt_{\bullet}(Y,\spin)$ is zero for degrees negative enough and the map $i_*$ is an isomorphism in degrees high enough. This implies that for some $\alpha\geq \beta\geq \gamma$ rational numbers with the same fractional part we have that
\begin{equation*}
i_*\left(\HSb_{\bullet}(Y,\spin)\right)\cong \mathcal{S}^+_{\alpha,\beta,\gamma}
\end{equation*}
as absolutely graded $\Rin$-modules. A very interesting special case is when $Y$ is an actual integral homology sphere. In that case the correction terms are integral lifts of the Rokhlin invariant, and are invariant under homology cobordism. Furthermore
\begin{equation*}
\beta(-Y)=-\beta(Y),
\end{equation*}
where $-Y$ denotes the manifold with opposite orientation, and using these properties one can show that the long standing Triangulation conjecture is false (see \cite{Man2} and \cite{Lin}).

\vspace{0.7cm}

\textbf{Construction of monopole Floer homology.} We quickly describe the construction of monopole Floer homology. Equip $Y$ with a riemannian metric. A spin$^c$ structure on $Y$ is given by a rank two hermitian bundle $S\rightarrow Y$ together with a Clifford multiplication
\begin{equation*}
\rho: TY\rightarrow \mathrm{Hom}(S,S)
\end{equation*}
i.e. a bundle map with the property that
\begin{equation*}
\rho(v)^2=-\|v\|^2\mathrm{Id}_S\text{ for each }v\in TY.
\end{equation*}
We can define the configuration space $\mathcal{C}(Y,\spin)$ consisting of pairs $(B,\Psi)$ where
\begin{itemize}
\item $B$ is a spin$^c$ connection on $S$, i.e. a connection compatible with the Clifford action;
\item $\Psi$ is a spinor, i.e. a section of $S$.
\end{itemize}
The group of automorphism of the Clifford bundle $\mathcal{G}(Y,\spin)$ is given by the set of maps $u$ from $Y$ to $S^1$, and it acts on the configuration space as
\begin{equation*}
u\cdot(B,\Psi)=(B-u^{-1}du, u\cdot\Psi).
\end{equation*}
A configuration with $\Psi\neq 0$, called \textit{irreducible}, has trivial stabilizer, while a configuration $(B,0)$, called \textit{reducible}, has stabilizer $S^1$ given by the constant gauge transformations. The functional on the configuration space used to define the Floer chain complex is the Chern-Simons-Dirac functional, given after a choice of a base spin$^c$ connection $B_0$ by
\begin{equation*}
\mathcal{L}(B,\Psi)=-\frac{1}{8}\int_Y (B^t-B^t_0)\wedge (F_{B^t}+F_{B_0^t})+\frac{1}{2}\int_Y \langle D_B\Psi,\Psi\rangle \mathrm{dvol}.
\end{equation*}
Here $B^t$ denotes the connection induced on the determinant line bundle $\Lambda^2S$, and $D_B$ is the Dirac operator associated to $B$. When $c_1(\spin)$ is not torsion, this functional is invariant only under the identity component of the gauge group, and is well defined with values in $\mathbb{R}/2\pi^2\mathbb{Z}$.
\par
From here, the early approaches to Seiberg-Witten Floer homology went on considering the flow of the $L^2$ formal gradient of the functional $\mathcal{L}$, given by the formula
\begin{equation*}
\mathrm{grad}\mathcal{L}(B,\Psi)=\left( \left(\frac{1}{2}\ast F_{B^t}+\rho^{-1}(\Psi\Psi^*)_0\right)\otimes 1_S, D_B\Psi\right)
\end{equation*}
on the space of irreducible configurations modulo gauge. This caused invariance issues, and the loss of important informations carried by the reducibles. Kronheimer and Mrowka's approach solved this by passing to the blown-up configuration space $\mathcal{C}^{\sigma}(Y,\spin)$. This consists of triples $(B,r,\psi)$ where
\begin{itemize}
\item$B$ is again a spin$^c$ connection;
\item $r$ is a non-negative real number;
\item $\psi$ is a spinor with unit $L^2$ norm.
\end{itemize}
This operation is essentially passing to polar coordinates in the spinor component, and there is a canonical blow-down map $\pi$ given by
\begin{equation*}
(B,r,\psi)\mapsto (B,r\psi).
\end{equation*}
The action of the gauge group on this space is free, and the quotient $\mathcal{B}^{\sigma}(Y,\spin)$ is an infinite dimensional manifold with boundary $\partial \mathcal{B}^{\sigma}(Y,\spin)$ consisting of the reducible configurations. Furthermore the gradient of the Chern-Simons-Dirac functional naturally extends to the blown-up configuration space as the vector field
\begin{equation}\label{blowupgradient}
(\mathrm{grad}\mathcal{L})^{\sigma}(B,r,\psi)=\left(\left(\ast \frac{1}{2}F_{B^t}+r^2\rho^{-1}(\psi\psi^*)_0\right)\otimes 1_S, \Lambda(B,r,\psi)r, D_B\psi-\Lambda(B,r,\psi)\psi)\right),
\end{equation}
where $\Lambda(B,r,\psi)$ is the real valued function $\langle D_B\psi,\psi\rangle_{L^2(Y)}$. Monopole Floer homology is then defined applying the usual ideas of Morse-Witten homology to this vector field on the space $\mathcal{B}^{\sigma}(Y,\spin)$.
\par
There are a few observations to be made. First of all, the vector field $(\mathrm{grad}\mathcal{L})^{\sigma}$ is not the gradient of any natural function on $\mathcal{C}^{\sigma}(Y,\spin)$ with respect to any natural metric. Furthermore it is tangent to the boundary $\partial\mathcal{B}^{\sigma}(Y,\spin)$, so the latter is invariant under the flow. This implies that there are two kinds of reducible critical points: the \textit{stable} ones, for which the Hessian in the direction normal to the boundary is positive, and the \textit{unstable} ones. The unstable manifold of a stable critical point is entirely contained in the boundary, and similarly for the stable manifold of an unstable critical point. In particular a non empty space of trajectories from a stable point to an unstable point is \textit{never} transversely cut out. We say that such a pair is \textit{boundary obstructed}. This implies for example that a one dimensional family of unparametrized trajectories between two irreducible critical points may break into three components, the middle one being a boundary obstructed trajectory (see Figure \ref{obstructed}).
\begin{figure}
  \centering
\def\svgwidth{0.5\textwidth}
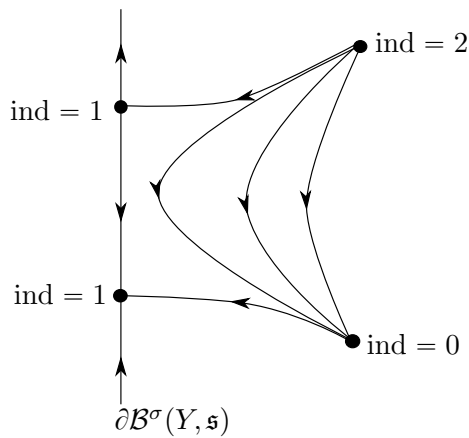
    \caption{A one dimensional family of trajectories limiting to a broken trajectory with three components.}\label{obstructed}
\end{figure} 

Nevertheless for generic perturbations one can define three chain complexes computing the three relevant homologies as follows. Denote by $C^o, C^u, C^s$ the $\ztwo$-vector spaces generated by the irreducible, unstable and stable critical points. One can define the linear maps
\begin{align*}
\partial^o_o&: C^o\rightarrow C^o\\
\partial^o_s&: C^o\rightarrow C^s\\
\partial^u_o&: C^u\rightarrow C^o\\
\partial^u_s&: C^o\rightarrow C^o
\end{align*}
obtained by counting irreducible trajectories in zero dimensional moduli spaces between critical points of a specified type, and similarly
\begin{align*}
\bar{\partial}^s_s&: C^s\rightarrow C^s\\
\bar{\partial}^s_u&: C^s\rightarrow C^u\\
\bar{\partial}^u_s&: C^u\rightarrow C^s\\
\bar{\partial}^u_u&: C^u\rightarrow C^u
\end{align*}
obtained by counting reducible trajectories. Notice that the maps $\partial^u_s$ and $\bar{\partial}^u_s$ count points in different moduli spaces. The three Floer chain complexes are then defined as the $\ztwo$-vector spaces
\begin{equation*}
\check{C}_*=C^o\oplus C^s,\qquad \hat{C}_*=C^o\oplus C^u, \qquad \bar{C}_*=C^s\oplus C^u
\end{equation*}
equipped respectively with the differentials
\begin{equation*}
\check{\partial}=
\begin{bmatrix}
\partial^o_o & \partial^u_o\bar{\partial}^s_u\\
\partial^o_s & \bar{\partial}^s_s+\partial^u_s\bar{\partial}^s_u
\end{bmatrix},\qquad
\hat{\partial}=
\begin{bmatrix}
\partial^o_o & \partial^u_o\\
\bar{\partial}^s_u\partial^o_s & \bar{\partial}^s_u\partial^u_s
\end{bmatrix},\qquad
\bar{\partial}=
\begin{bmatrix}
\bar{\partial}^s_s & \bar{\partial}^s_u \\
\bar{\partial}^u_s & \bar{\partial}^u_u
\end{bmatrix}.
\end{equation*}
The Floer homology groups are the homology groups of these complexes, and are independent (in a functorial way) of the choices made.
\\
\par
We can give a nice description of the critical points in the blow up in the generic case. Recall first of all that the compactness properties of the Seiberg-Witten equations imply that for a generic perturbation there are only finitely many critical points downstairs. The blow down map is a diffeomorphism on the irreducibles, so we have a one-to-one correspondence in this case, so that the irreducible critical points upstairs consist of a finite number of points. On the other hand, it is easy to see from equation (\ref{blowupgradient}) that for each reducible critical point downstairs gives rise to countably collection of stable and unstable critical points. Indeed, each of these corresponds to the quotient by the $S^1$ action of the unit sphere of an eigenspace of the Dirac operator. These are generically all (complex) one dimensional, and the eigenvalues form a discrete sequence of real numbers which is infinite in both directions. The computation of the Floer groups for $S^3$ then readily follows because the positive scalar curvature implies that there are no irreducible critical points for perturbations small enough.
\\
\par
Finally, given a cobordism $W$ between $Y_0$ and $Y_1$ the induced map is defined by counting solutions to the four dimensional Seiberg-Witten equations (in the blow-up of the configuration space) on the manifold $W^*$ obtained by adding cylindrical ends. This generally involves infinitely many moduli spaces, and their sum only makes sense after a suitable completion of the groups with respect to some negative filtration, which we denote using the bullet. The map induced by $W$ and the class $U^d\in\ztwo[[U]]$ is defined by considering the manifold $W_p^*$ with an additional incoming end of the form $(-\infty,0]\times S^3$ obtained by removing a ball (such that the metric is a product near the boundary) and adding a cylindrical end. We can choose the metric such that $S^3$ has positive scalar curvature, and the map is obtained by counting the solutions to the Seiberg-Witten equations on $W_p^*$ that converge to the $d$th unstable critical point at this additional end.

\vspace{0.7cm}

\textbf{The case of a self-conjugate spin$^c$ structure.} In the case the spin$^c$ structure is self-conjugate, or equivalently the spin$^c$ structure is induced by a genuine spin structure, we have two extra features:
\begin{itemize}
\item a preferred base connection $B_0$ with $B_0^t$ flat, the spin connection;
\item a quaternionic structure $j$ on the Clifford bundle $S$, i.e. a complex antilinear automorphism such that $j^2=-\mathrm{Id}$.
\end{itemize}
The latter follows from the observation that $\mathrm{Spin}(3)$ can be identified with $\mathrm{SU}(2)$ and the spinor representation is just the usual action on $\mathbb{C}^2=\mathbb{H}$. We will write the action of $j$ and the multiplication by complex numbers from the right. These two features are compatible in the sense that the Dirac operator $D_{B_0}$ is quaternionic linear. In this case the configuration space $\mathcal{C}(Y,\spin)$ comes with diffeomorphism $\jmath$ given by
\begin{equation*}
\jmath\cdot(B_0+b,\Psi)=(B_0-b,\Psi\cdot \jmath).
\end{equation*}
This induces an involution on the moduli space of configurations $\mathcal{B}(Y,\spin)$. Its only fixed points are the equivalence classes $[B,0]$ with $B$ the spin connection of a spin structure inducing the spin$^c$ structure $\spin$. There are $2^{b_1(Y)}$ such spin structures: for example on $S^2\times S^1$ the two spin structures both induce the only torsion spin$^c$ structure. Similarly, there is an induced involution (still denoted by $\jmath$) on the blown-up moduli space of configurations $\mathcal{B}^{\sigma}(Y,\spin)$ which is fixed point-free.
\\
\par
The main idea is then to do Floer theory in a $\jmath$-invariant fashion, and compute the homology of the quotient $\mathcal{B}^{\sigma}(Y,\spin)/\jmath$. This is done by producing a (Morse-Bott) chain complex with a natural $\mathbb{Z}/2\mathbb{Z}$ action, and consider the homology of the invariant subchain complex. To do this we need to restrict to perturbations that preserve the symmetry of our set up. The generic picture will then be the following. Irreducible critical points are not fixed points of the action downstairs, so they will still be a finite number of points, and furthermore they come in pairs related by the action of $\jmath$. The same thing happens for reducible critical points for which the connection is not spin, with the action of $\jmath$ exchanging the two towers of reducible critical points. In the case when the connection is spin, the perturbed Dirac operator will still be quaternionic, hence the eigenspaces will always even dimensional over the complex numbers. In particular, the perturbation will not be regular in the usual sense. Nevertheless, generically the eigenspaces will be all two dimensional, and we will obtain a copy of $S^2$ as critical submanifold for each of them. This is the quotient of the unit sphere of the eigenspace by the action of $S^1$, which is just a Hopf fibration. Finally the action of $\jmath$ on $S^2$ is just the antipodal map.
\\
\par
For a generic perturbation the critical manifolds described above will be Morse-Bott, meaning that the Hessian is non degenerate in the normal directions. The three Morse-Bott chain complexes computing the homology introduced in Chapter $3$ of \cite{Lin} are defined following the framework of \cite{Fuk}. The underlying vector space is the direct sum of some variants of the singular chain complexes of the critical submanifolds. The differential combines the singular differential together with terms involving different critical submanifolds. In particular given critical submanifolds $[\Cr_{\pm}]$, there are evaluation maps on the compactified moduli spaces of trajectories connecting them
\begin{equation*}
\mathrm{ev}_{\pm}:\breve{M}^+([\Cr_-],[\Cr_+])\rightarrow [\Cr_{\pm}]
\end{equation*}
sending a trajectory to its limit points. Then a chain
\begin{equation*}
f:\sigma\rightarrow [\Cr_-]
\end{equation*}
gives rise (under suitable transversality hypothesis) to the chain
\begin{equation*}
\mathrm{ev_+}: \sigma\times \breve{M}^+([\Cr_-],[\Cr_+])\rightarrow [\Cr_+]
\end{equation*}
where the underlying space is the fibered product under the maps $f$ and $\mathrm{ev}_-$. The total differential of $\sigma$ is then defined to be the sum of its singular differential and all these chains obtained via fibered products. The proofs in this new framework carry over with the same formulas as the usual one: identities relating zero dimensional moduli spaces coming from boundaries of one dimensional spaces are now identities (at the chain level) of chains arising as the codimension one strata of fibered products as above. Of course, we need to consider chains $\sigma$ some class of geometric objects so that the fibered products with the moduli spaces remain in that class. This is a delicate point because in our context the compactified moduli spaces are not in general manifolds with corners, neither combinatorially nor topologically. In Section $3.1$ of \cite{Lin} we introduce a suitable class of such objects called $\delta$-chains, and define a modified singular chain complex for a smooth manifold obtained by quotienting out a class of degenerate $\delta$-chains. These degenerate chains are essentially chains whose image is contained in the image of a chain of strictly smaller dimension (see also \cite{Lip}). This construction leads to a well defined homology theory for smooth manifolds satisfying the Eilenberg-Steenrod axioms.
\par
These chain complexes come with a natural involution (given by composition with the involution $\jmath$), and the $\Pin$-monopole Floer homology groups are defined as the homologies of the invariant subcomplexes. Finally by exploiting the positive scalar curvature we see that the $\Pin$-monopole Floer homology groups of the three sphere are isomorphic (after grading reversal) to the claimed ones.
\\
\par
While transversality for the three-dimensional equations can be achieved using the equivariant counterpart of the three dimensional perturbations used in \cite{KM}, in the case of the map induced by a cobordism $W$ new perturbations (defined only in the blown up configuration space of the cobordism) need to be considered. This is because the three dimensional perturbations always vanish at the fixed points of the involution $\jmath$. In particular if we are only considering them a spin connection on a cobordism always gives rise to a reducible solution, and this might not be regular simply for index reasons. In Section $4.2$ of \cite{Lin} we define $\Pin$-equivariant $\textsc{asd}$-perturbations. These are maps
\begin{equation}\label{squaredproj}
\hat{\omega}:\mathcal{C}^{\sigma}(X)\rightarrow L^2(X;i\mathfrak{su}(2))
\end{equation}
which are gauge invariant, $\jmath$-equivariant (where $\jmath$ acts on the right hand side as the multiplication by $-1$) and satisfy suitable analytical properties. Furthermore, they form a collection large enough so that we can achieve transversality while preserving the $\Pin$-symmetry.

\vspace{1cm}
\section{Some additional background results}\label{blowups}
In this section we discuss some aspects of $\Pin$-monopole Floer homology which have not been treated in \cite{Lin} and will be central to the discussion in the present work. In particular, we discuss the blow-up formula and describe the reducible moduli spaces on special spin cobordisms with $b_2^+=1$.
\\
\par
Recall that the \textit{blow-up} of a smooth $4$-manifold $X$ (possibly with boundary) is the smooth four manifold
\begin{equation*}
\widetilde{X}=X\hash \CPbar.
\end{equation*} 
We will denote by $E$ the exceptional divisor, which is a smooth embedded sphere of self-intersection $-1$. We define $\spin_k$ to be the spin$^c$ structure on $\CPbar$ such that
\begin{equation*}
\langle c_1(\spin_k), [E]\rangle=2k-1.
\end{equation*}
In particular the conjugate of $\spin_k$ is $\spin_{1-k}$. Given a spin$^c$ structure $\spin$ on $X$, we define $\spin\hash\spin_k$ to be the spin$^c$ structure on $\widetilde{X}$ that restricts to $\spin$ on $X$ and to $\spin_k$ on $\CPbar$. It is shown in Section $39.3$ of \cite{KM} that in monopole Floer homology we have the identity
\begin{equation*}
\HMt_{\bullet}(\widetilde{X},\spin\hash\spin_k)=\HMt_{\bullet}(U^{k(k-1)/2}\mid X,\spin).
\end{equation*}
The aim of this section is to prove the counterpart of this in the case of $\Pin$-monopole Floer homology. It is important to notice that the blow up $\widetilde{X}$ is never a spin manifold, as it carries a homology class $[E]$ with odd self-intersection.

\begin{prop}\label{blowup}
Let $X$ be a cobordism. If $\spin$ a self conjugate spin$^c$ structure, then
\begin{equation}\label{blowupformula}
\HSt_{\bullet}(\tilde{X},[\spin\hash \spin_k])=
\begin{cases}
0 &\text{if } k\equiv 0,1\pmod{4}\\
\HSt_{\bullet}(Q^2V^{\lfloor k(k-1)/4\rfloor}\mid X, [\spin]) &\text{otherwise.}
\end{cases}
\end{equation}
If the the spin$^c$ structure on $X$ is \textit{not} self-conjugate, then 
\begin{equation*}
\HSt_{\bullet}(\tilde{X},[\spin\hash \spin_k])=0
\end{equation*}
for every $k$. The same statement holds for the from and bar versions.
\end{prop}

As it will be clear from the proof, it is natural to interpret the two cases in equation (\ref{blowupformula}) by considering the parity of $k(k-1)/2$: in the first case it is even, in the latter it is odd. 
\begin{proof}
The result can be proved by a neck stretching argument close to the ones in Section $3.2$ of \cite{Lin}, and in particular the construction of the $\Rin$-module structure. We focus on the case of a self-conjugate spin$^c$ structure $\spin$ first. Consider the separating three sphere $S^3$ along which the connected sum is performed. We can suppose that the metric on this $S^3$ is the standard round one, and it is a product in a neighborhood. Hence we can define a one dimensional family of metrics parametrized by $T\in[0,\infty)$ by adding longer and longer necks of the form $[0,T]\times S^3$. By considering the compactified moduli spaces of solutions parametrized by this family of metrics, we can construct a chain homotopy between the map defining $\HSt_{\bullet}(\widetilde{X},[\spin\hash \spin_k])$ (which is the map corresponding to the stratum $T=0$), and a new chain map corresponding to an additional stratum at $T=\infty$ which we describe in detail. Intuitively, at $T=\infty$ the cobordism is decomposed in two pieces, and the solutions converge to solutions on each of the two pieces. On one hand we have the moduli spaces for the self conjugate spin$^c$ structure on $X$ with an additional cylindrical end $(0,\infty]\times S^3$ (denoted by $X_p^*$ in \cite{Lin}) with a $\Pin$-equivariant perturbation on it. Recall from Section \ref{review} that the map induced by the cobordism $X$ and an element $x$ in $\Rin$ is given by considering the moduli spaces on $X_p^*$ which are asymptotic on this additional end to a given representative of $x$ in $\HSf_{\bullet}(S^3)$. The latter can be realized as a $\jmath$-invariant chain in a critical submanifold.
Furthermore, we have the moduli spaces on a punctured $\CPbar$, and the perturbation on the cylindrical end is $\Pin$-equivariant. Denote the critical submanifold corresponding to the $i$th negative eigenvalue $[\Cr_{-i}]$. Then for dimensional reasons (as the critical submanifold is $2$-dimensional) the only interesting moduli space for our purposes when dealing with the spin$^c$ structure $\spin_k$ is
\begin{equation*}
M_k=M^+\left((\CPbar)^*,[\Cr_{\lfloor k(k-1)/4\rfloor-1}],\spin_k\right),
\end{equation*} 
as the others have dimension too big so are degenerate as discussed in Section $1$. Following the proof of Fr\o yshov's theorem (Section $39.1$ in \cite{KM}), as there are no $L^2$ anti-self-dual forms on $(\CPbar)^*$ there is only one reducible connection. Furthermore as the solutions converge on the cylindrincal end to an unstable configuration they are necessarily reducible. If $k\equiv 0,-1\pmod{4}$ then $M_k$ is the projectivization of to the two dimensional kernel of the corresponding Dirac equation. In particular it is a two dimensional sphere of reducible solutions, and the evaluation map is either constant or a diffeomorphism. On the other hand if $k\not\equiv 0,-1\pmod{4}$ $M_k$ is the projectivization of the one dimensional kernel, so it consists of a single point.
\par
Going back to the map induced by $\widetilde{X}$ on $\HSt_{\bullet}$, we see via the chain homotopy provided by the neck stretching argument that this is equivalent to the map induced by $X$ together with a specified cohomology class in $\Rin$. In the case $k\equiv 0,-1\pmod{4}$ this is defined by the invariant chain given by the union $M_k\cup M_{1-k}$. As again the two chains are related by the action of $\jmath$, we have that either they both represent a generator of the top homology or they both represent the zero class. In any case their union defines the zero class in the homology of the invariant subcomplex, so the induced map is zero. In the other case $k\not\equiv 0,-1\pmod{4}$ the cohomology class is given by $Q^2V^{\lfloor k(k-1)/4\rfloor}$, as a pair of antipodal points is a generator of the invariant subcomplex of the critical submanifold in dimension zero.
\par
There is a slight subtlety regarding perturbations in this proof. Because the blow up $X\hash \CPbar$ is not spin, when performing the stretching argument we can use a non equivariant perturbation on the neck while preserving equivariance. On the other hand, at the limit we need to use an equivariant perturbation on the manifold with cylidrical ends $X^*_p$ in order for the map to make sense. From this it is clear that in the case in which the spin$^c$ structure on $X$ is \textit{not} self conjugate, we can use non equivariant perturbations all the way and obtain the second part of the result, as in this case the standard argument works.
\end{proof}

\vspace{0.8cm}

The second fact we want to point out is that unlike the usual counterpart in monopole Floer homology, the map $\HSb_{\bullet}(W)$ induced by a cobordism $W$ with $b_2^+\geq1$ is not necessarily zero. In fact in a special case we have the following characterization of the moduli spaces. 
\begin{prop}\label{modulib2p}
Let $W$ be a cobordism between rational homology spheres with $b_1(W)=0$ and $b_2^+=1$ and a consider self conjugate spin$^c$ structure $\spin_0$. Let $A_0$ be the corresponding spin$^c$ connection. Suppose that for the fixed perturbations at the end there is only one reducible solution. Then there exists a regular perturbation so that the one dimensional reducible moduli spaces consist of generators of the one dimensional $\jmath$-invariant homology of the critical submanifolds they evaluate to (which are copies of $\mathbb{C}P^1$). 
\end{prop}
If we restrict to $\Pin$-equivariant three dimensional perturbations (which we recall are introduced in a collar $I\times \partial W$ of the boundary of $W$, see Chapter $24$ in \cite{KM}), the reducible configuration in the blow-down $[A_0,0]$ is always a solution, because the perturbations vanish at the fixed points of the involution. In the blow-up, the reducible moduli spaces $M^+([\Cr_-],[\Cr_+])$ lying over $[A_0,0]$ and such that the linearization of the equations has index $1$ are in the best scenation a copy of $\mathbb{C}P^1$, given by the one dimensional (over the quaternions) kernel of the perturbed Dirac operator $D^+_{A_0,\hat{\mathfrak{p}}}$. In particular, the moduli spaces cannot be transversely cut out.

\begin{proof}
We will introduce a suitable regular perturbation for which we can explicitly describe the moduli spaces in play. As $b_2^+=1$, the operator $d^+$ does not have dense image, we can find a smooth compactly supported form self-dual form $\omega_0^+$ not contained in it. We can suppose without loss of generality that the support of $\omega_0^+$ is contained in $W$ away for the collar of the boundary where three dimensional perturbations are applied. This implies that the form $\omega_0^+$ is not in the image of any of the linearizations of the perturbed anti-self-duality operators. We claim that there exists a smooth function on the (completion in the $L^2_k$ norm of the) blown up configuration space 
\begin{equation*}
\eta: \mathcal{C}^{\sigma}_k(W,\spin_W)\rightarrow \mathbb{R}
\end{equation*}
with the following properties:
\begin{enumerate}
\item the map is gauge invariant and $\jmath$-equivariant where $\jmath$ acts on $\mathbb{R}$ as $-1$;
\item the restriction of the map $\eta$ to the reducible moduli space lying over $[A_0,0]$ is transverse to zero;
\item the $L^2(W;i\mathfrak{su}(2))$-valued map $\eta\cdot \rho_Z(\omega^+_0)$ is a $\Pin$-equivariant $\textsc{asd}$-perturbation in the sense of Section $4.2$ of \cite{Lin}, so that it can be legitimately used to perturb the equations.
\end{enumerate}
The map $\eta$ can be easily constructed from a given $\Pin$-equivariant $\textsc{asd}$-perturbation $\hat{\omega}$ via the $L^2$ projection to the line spanned by a single configuration. The proofs in Section $4.2$ in \cite{Lin} carry over to show that the perturbation $\eta\cdot \rho(\omega^+_0)$ has the required analytic properties. 
\par
We can then describe the one dimensional moduli spaces of reducible solutions as follows. The new equations for a configuration reducible configuration $(A,0,\phi)$ is (a small perturbation of)
\begin{align*}
D_A^+\phi&=0\\
\rho(F_{A^t}^+)-\eta(A,0,\rho)\omega_0^+&=0.
\end{align*}
As $\omega_0^+$ generates the cokernel of $d^+$, the last equation implies $\eta(A,0,\phi)$ is zero and $(A,0,\phi)$ is a reducible solution for the equations without the additional perturbation. Hence the space of reducible solutions is identified as the zero locus of
\begin{equation*}
\eta:\mathbb{C}P^1\rightarrow \mathbb{R}.
\end{equation*}
This is transversely cut out by condition $(2)$ and consists of an odd number of circles by condition $(1)$, hence the result follows.

\end{proof}

\begin{proof}[Proof of Theorem \ref{108}]
This result follows in the same way as the proof of monotonicy of Manolescu's invariants under negative definite cobordisms, see Section $4.4$ of \cite{Lin}. If $b_2^+=1$, after reducing to the case of $b_1=0$ via surgery, the description above tells us that for a self-conjugate spin$^c$ structure $\spin_0$ the map
\begin{equation*}
\HSb_{\bullet}(W,\spin_0):\HSb_{\bullet}(Y_0,\spin_0)\rightarrow \HSb_{\bullet}(Y_1,\spin_1)
\end{equation*}
is as follows. After fixing isomorphisms of graded $\Rin$-modules
\begin{equation*}
\HSb_{\bullet}(Y_i,\spin_i)\cong \mathcal{S}\langle d_i\rangle,
\end{equation*}
where for some choice of $d_i$ the angular bracket denote a global grading shift of $d_i$, the map is identified to be the multiplication by $QV^k$ where the integer $k$ is determined by the intersection form of the cobordism, and the statement follows.
\par
In the case $b_2^+=2$ an analogous characterization of the reducible moduli space holds: it consists of a number of points congruent to $2$ modulo $4$. The proof of this characterization follows that of Proposition \ref{modulib2p}, and we briefly sketch it here. The cokernel of $d^+$ now has dimension two, and we can construct an analogous $\jmath$ perturbation with values in cokernel. When restricted to the reducible unperturbed moduli space, this has the form of
\begin{equation}\label{ciao}
\eta: \mathbb{C}P^1\rightarrow \mathbb{R}^2.
\end{equation} 
and the (transverse) zero set consists of a number of points congruent to $2$ modulo $4$ by $\jmath$ equivariance. The same argument as in Proposition \ref{modulib2p} implies then that the reducible solutions are identified with the zero set of (\ref{ciao}) and they are transversely cut out. This identifies the map on the bar version as the multiplication by $Q^2V^k$, and the result follows.
\end{proof}

\vspace{1cm}
\section{The exact triangle}\label{surgery}

This section is dedicated to the proof of the main result of the paper, Theorem \ref{exacttr} in the Introduction.
We start by reviewing in detail the framework of the surgery exact triangle (see also \cite{KMOS} and Chapter $42$ in \cite{KM}). Suppose we are given a knot $K$ inside a three manifold $Y$, and let $Z$ be the manifold with torus boundary $\partial Z$ obtained by removing a tubular neighborhood of it. Let $\mu_1,\mu_2,\mu_3$ be closed simple curves on $\partial Z$ with the property that the intersection numbers satisfy
\begin{equation}\label{curves}
\mu_1\cdot \mu_2=\mu_2\cdot\mu_3=\mu_3\cdot \mu_1=-1. 
\end{equation}
We can then obtain the three manifolds $Y_1, Y_2$ and $Y_3$ by Dehn filling along these curves. We extend this definition periodically so that for example $\mu_{n+3}=\mu_n$.
\par
This construction behaves well with four dimensional topology in the following sense. There is an elementary cobordism $W_n$ from $Y_n$ to $Y_{n+1}$ obtained by attaching a single $2$-handle $D^2\times D^2$ to $[0,1]\times Y_n$ along $\{1\}\times K$ with framing $\mu_{n+1}$. In this case the knot $K\subset Y_{n+1}$ with framing $\mu_{n+2}$ can be identified with the boundary $\{0\}\times S^1$ of the cocore of the attached handle with framing $-1$ relative to the cocore $\{0\}\times D^2$. In $W_n$ there is a closed $2$-cycle $\Sigma_n$ defined as follows. On $\partial Z$ there is the distinguished simple closed curve $\sigma$ whose homology class generates the kernel of $H_1(\partial Z)\rightarrow H_1(Z)$, so in particular there is a surface $T\subset Z$ with boundary $\sigma$. The closed cycle $\Sigma_n$ is obtained as the union of $T$, the core of the $2$-handle and a piece in the solid torus $\partial D^2\times D^2$.
\\
\par
Suppose we are in the case $Z$ is a knot complement in $S^3$. Then there is canonical pair of curves $m$ and $l$ is $\partial Z$, namely to be the classical meridian and longitude of the knot. These are oriented so that $m\cdot l=-1$. In this case, any oriented simple closed curve $[\mu]$ can be described up to isotopy by its homology class
\begin{equation*}
[\mu]=p[m]+q[l]
\end{equation*}
with $(p,q)$ relatively prime. Forgetting about the orientation, this can be recorded by the ratio
\begin{equation*}
r=p/q\in \mathbb{Q}\cup \{\infty\}.
\end{equation*}
Triples of curves satisfying the relations (\ref{curves}) come from triples of pairs
\begin{equation*}
(p_1,q_1),\quad (p_2,q_2),\quad (p_3,q_3)
\end{equation*}
satisfying the conditions
\begin{equation*}
-p_nq_{n+1}+p_{n+1}q_n=-1,
\end{equation*}
where as usual the subscripts are interpreted modulo $3$. Very interesting cases are given by the slopes
\begin{equation*}
r_1=0,\quad r_2=1/(q+1),\quad r_3=1/q
\end{equation*}
for $q\in \mathbb{Z}$ and
\begin{equation*}
r_1=p, \quad r_2=p+1,\quad r_3=\infty
\end{equation*}
for $p\in \mathbb{Z}$. In general, if none of the slopes is zero, up to cyclic permutation and change of sign we can suppose that our triple looks like
\begin{equation*}
r_1=\frac{p}{q},\quad r_2=\frac{p+p'}{q+q'},\quad r_3=\frac{p'}{q'}
\end{equation*}
with the properties
\begin{equation*}
p,p'>0,\quad -p'q+pq'=-1.
\end{equation*}
In particular, exactly two of $p, p'$ and $p+p'$ are odd. Furthermore in this case the cobordism with positive definite intersection form is $W_3$ (see Section $42.3$ in \cite{KM}).
\\
\par
We now focus on the interactions of this construction with our invariants.
\begin{lemma}
Among the three cobordisms $W_1$, $W_2$ and $W_3$, exactly two are spin.
\end{lemma}
\begin{proof}
Because the cobordism $W_n$ is given by a two handle attachment and three manifolds are always spin, the fact that $W_n$ is spin (i.e. its second Stiefel-Whitney class is zero) is equivalent to the fact that the cycle $[\Sigma_n]$ has even self-intersection. To see that this holds, we use the discussion above, which can be generalized to any manifold $Z$ with torus boundary by taking $[l]$ to be a primitive element in the kernel of
\begin{equation*}
H_1(\partial Z)\rightarrow H_1(Z)
\end{equation*}
and $[m]$ any other curve such that $[m]\cdot[l]=-1$. In particular, while the slope of a curve $\gamma_n$ is not well defined (as it depends on the choice of the meridian $[m]$), the numerator $p_n$ is. The self intersection is up to sign just the product of the numerators $p_np_{n+1}$, so the result follows because exactly two of them are odd.
\end{proof}
On the other hand, the composition $W_n\cup_{Y_{n+1}} W_{n+1}$ is never spin. In fact it always contains a sphere $E_{n}$ with self intersection $-1$. This is given by the union of the core of the $2$-handle of $W_{n+1}$ and the cocore of the $2$-handle in $W_n$. From this description it also follows that the cobordism
\begin{equation*}
W_n\cup_{Y_{n+1}} W_{n+1}
\end{equation*}
between $Y_n$ and $Y_{n+2}$ is diffeomorphic to the opposite cobordism $\overline{W}_{n+2}$ blown up at a point. \textit{Without loss of generality} we will suppose from now on that the non spin cobordism is $W_3$.
\\
\par
We introduce the homological algebra needed for our purposes in an abstract setting. This is a slight variation of a standard triangle detection result in Floer homology, see Lemma $4.2$ in \cite{OSbr} or Lemma $5.1$ in \cite{KMOS}. Suppose we are given three chain complexes $C_1$, $C_2$ and $C_3$, and chain maps $f_1:C_1\rightarrow C_2$ and $f_2:C_2\rightarrow C_3$ such that the composition $f_2\circ f_1$ is homotopic to zero via a nullhomotopy $H_1$. We can form the ``iterated mapping cone'' ${C}$ whose underlying vector space is ${C}_3\oplus {C}_2\oplus {C}_1$ and whose differential is
\begin{equation*}
{\partial}=
\begin{pmatrix}
\partial_3 & f_2 & {H}_1 \\
0 & {\partial}_2 & {f}_1 \\
0 & 0& {\partial}_1
\end{pmatrix}.
\end{equation*}
This is a differential because $f_1$ and $f_2$ are chain maps and $H_1$ is a chain nullhomotopy for $f_2\circ f_1$. The following is the key lemma in homological algebra we need.
\begin{lemma}\label{triangledetection}
Suppose the homology of the iterated mapping cone $H_*(C,\partial)$ is trivial. Then there is a map $F_3: H_*(C_3)\rightarrow H_*(C_1)$ such that the triangle
\begin{center}
\begin{tikzpicture}
\matrix (m) [matrix of math nodes,row sep=2em,column sep=1.5em,minimum width=2em]
  {
  H_*(C_2) && H_*(C_3)\\
  &H_*(C_1) &\\};
  \path[-stealth]
  (m-1-1) edge node [above]{$(f_2)_*$} (m-1-3)
  (m-2-2) edge node [left]{$(f_1)_*$} (m-1-1)
  (m-1-3) edge node [right]{$F_3$} (m-2-2)  
  ;
\end{tikzpicture}
\end{center}
is exact.
\end{lemma}
We will discuss the naturality properties of this construction (in our specific case) in detail in the proof of Theorem \ref{exacttr}.
\begin{proof}
The proof of this result follows closely that of the standard triangle detection lemma. First we form the mapping cone $M_{f_1}$ of the chain map $f_1$, which is the chain complex with underlying vector space ${C}_1\oplus {C}_2$ and differential
\begin{equation*}
d=
\begin{pmatrix}
{\partial}_2 & {f}_1 \\
0 & {\partial}_1
\end{pmatrix}.
\end{equation*}
This is a differential because ${f}_1$ is a chain map. The short exact sequence of chain complexes
\begin{equation*}
0\rightarrow {C}_2\stackrel{i}{\longrightarrow} M_{f_1} \stackrel{p}{\longrightarrow} {C}_1\rightarrow 0,
\end{equation*}
where the maps are respectively the inclusion and the quotient, induces an exact triangle
\begin{center}\label{cone1}
\begin{tikzpicture}
\matrix (m) [matrix of math nodes,row sep=2em,column sep=1.5em,minimum width=2em]
  {
  H_*(C_2) && H_{*}(M_{f_1})\\
  &H_*(C_1) &\\};
  \path[-stealth]
  (m-1-1) edge node [above]{$i_*$} (m-1-3)
  (m-2-2) edge node [left]{$(f_1)_*$} (m-1-1)
  (m-1-3) edge node [right]{$p_*$} (m-2-2)  
  ;
\end{tikzpicture}
\end{center}
Similarly the iterated mapping cone fits in the short exact sequence of chain complexes
\begin{equation*}
0\rightarrow  {C}_3\rightarrow {C} \rightarrow {M_{f_1}}\rightarrow 0.
\end{equation*}
In particular we have a connecting homomorphism
\begin{equation*}
\delta: H_*({M_{f_1}})\rightarrow H_*(C_3)
\end{equation*}
which is induced by the chain map
\begin{equation}\label{connectingmap}
({f}_2+{H}_1):{M_{f_1}}={C}_2\oplus {C}_1\rightarrow {C}_3.
\end{equation}
The fact that $H_*(C)$ is trivial is equivalent to $\delta$ being an isomorphism. So in the triangle above we can replace the homology of the mapping cone $H_*(M_{f_1})$ with $H_*(C_3)$ using this isomorphism. Finally the connecting homomorphism $\delta$ is given by equation (\ref{connectingmap}), so can identify the horizontal map as $(f_2)_*$.
\end{proof}

\vspace{0.8cm}
We now describe how our problem fits in the framework of Lemma \ref{triangledetection}. We denote by $\check{C}_i$ the chain complex $\check{C}_*(Y_i)$ computing the \textit{to} version, and by
\begin{equation*}
\check{f}_i: \check{C}_i\rightarrow \check{C}_{i+1}
\end{equation*}
the chain map defining the map induced by the cobordism $W_i$. We denote by $\check{M}$ the mapping cone of $\check{f}_1$. Recall that the composition $W_2\circ W_1$ is $\overline{W}_3$ blown up at a point, and as $\overline{W}_3$ is not spin the induced map is zero by Proposition \ref{blowup}. There is in fact a natural chain homotopy to zero
\begin{equation*}
\check{H}_1: \check{C}_1\rightarrow \check{C}_3
\end{equation*}
defined as follows. The chain homotopy is constructed by considering a one dimensional family of metrics on the composite cobordisms $X_1=W_{2}\circ W_{1}$. There are two separating hypersurfaces $Y_2$ and $S_1$, the latter being the boundary of a neighborhood of the $(-1)$-sphere $E_1$. Choose a metric on $X_1$ such that $S_1$ has a metric obtained by flattening the round one near the Clifford torus $Y_2\cap S_1$. We can construct then the family of metrics $Q(S_1,Y_2)$ parametrized by $T\in\mathbb{R}$ given by inserting a cylinder $[-T,T]\times S_1$ normal to $S_1$  for $T$ negative, and a cylinder $[-T,T]\times Y_2$ normal to $Y_2$ for $T$ positive. Following the notation of \cite{KMOS} we will use the letter $Q$ to denote one dimensional families of metrics. It will be clear from the context whether we are using this letter to indicate this or the element of $\Rin$. We can define the moduli spaces parametrized by the family of metrics
\begin{equation*}
M_z([\Cr_-], X_1^*, [\Cr_+])_Q.
\end{equation*}
This will be a smooth manifold for a generic choice of $\Pin$-equivariant perturbation. As in the proof of Proposition \ref{blowup}, this can be compactified to a moduli space
\begin{equation*}
M_z^+([\Cr_-], X_1^*, [\Cr_+])_{\bar{Q}}
\end{equation*}
obtained by considering both broken trajectories and the fibered products of the compactified moduli spaces on the manifolds with (possibly more than two) cylindrical ends one obtains for $T=\pm\infty$, namely
\begin{equation*}
W_1^*\amalg W_2^*
\end{equation*}
for $T=+\infty$ and 
\begin{equation*}
B_1^*\amalg Z_1^*
\end{equation*}
for $T=-\infty$. Here $Z_1$ is a neighborhood of the exceptional divisor, and $B_1$ can be identified with $\overline{W}_3$ with a ball removed. Similarly one can consider the compactified moduli spaces consisting entirely of reducibles $M_z^{\mathrm{red}+}([\Cr_-], X_1^*, [\Cr_+])_{\bar{Q}}$. By taking fibered products with these moduli spaces one can define the $\jmath$-invariant linear map
\begin{equation*}
H^o_o: C^o(Y_1)\rightarrow C^o(Y_3)
\end{equation*}
and similarly its companions $H^o_s, H^u_o, H^u_s, \bar{H}^s_s, \bar{H}^s_u, \bar{H}^u_s$ and $\bar{H}^u_u$. We then define the map
\begin{equation*}
\check{H}_1=
\begin{bmatrix}
H^o_o & H^u_o\bar{\partial}^s_u+m^u_o(W_2)\bar{m}^s_u(W_1)+\partial^u_o\bar{H}^s_u\\
H^o_s & \bar{H}^s_s+ H^u_s\bar{\partial}^s_u+ m^u_s(W_2)\bar{m}^s_u(W_1)+\partial^u_s\bar{H}^s_u
\end{bmatrix}.
\end{equation*}
by the same formula that defines the chain map proving the composition formula (see Chapter $26$ in \cite{KM} and Section $3.3$ in \cite{Lin}). Here the maps $m^*_*(W_*)$ are the components defining the chain maps $\check{f}_1$ and $\check{f}_2$.
\par
In fact the same construction applies to the other two composites to give rise to the maps $\check{H}_2$ and $\check{H}_3$, and we have the following result.
\begin{lemma}\label{homotopies}
The map $\check{H}_1$ satisfies the identity
\begin{equation*}
\check{\partial}\circ\check{H}_1+\check{H}_1\circ \check{\partial}=\check{f}_2\circ \check{f}_1.
\end{equation*}
For $n=2,3$ the map $\check{H}_n$ is a chain homotopy between the composite $\check{f}_{n+1}\circ \check{f}_n$ and a chain map
\begin{equation*}
\check{g}_n: \check{C}_{\bullet}(Y_n)\rightarrow \check{C}_{\bullet}(Y_{n+2})
\end{equation*}
computing the blown up map as in Proposition \ref{blowup} in Section \ref{blowups}.
\end{lemma}

\begin{proof}
One just has to identify the contributions of the various codimension one strata of the moduli spaces $M_z^+([\Cr_-], X_1^*, [\Cr_+])_{\bar{Q}}$. The closure of the union of the codimension one strata of $M_z^+([\Cr_-], X_1^*, [\Cr_+])_{T}$ for $T$ finite corresponds to the left hand side. The moduli space $M_z^+([\Cr_-], X_1^*, [\Cr_+])_{+\infty}$ consists of the fibered products of the moduli spaces used to define the chain map on the right hand side. The moduli spaces $M_z^+([\Cr_-], X_1^*, [\Cr_+])_{-\infty}$ define the zero map at the chain level. This is a manifestation at the chain level of the phenomenon underlying the proof of Proposition \ref{blowup} in the non self-conjugate case. The case of the other cobordisms is analogous and follows the proof of Proposition \ref{blowup}.
\end{proof}

\vspace{0.5cm}
We can then form as in Lemma \ref{triangledetection} an iterated mapping cone $\check{C}$ of the three chain complexes $\check{C}_1,\check{C}_2, \check{C}_3$, the chain maps $\check{f}_1, \check{f}_2$ and the chain nullhomotopy $\check{H}_1$. As that lemma states, the main result underlying the existence of an exact triangle is the following.
\begin{prop}\label{mainiso}
The homology of the iterated mapping cone $\check{C}$ is zero.
\end{prop}
To prove this assertion, we will construct a chain map $\check{\varphi}$ from $\check{C}$ to itself homotopic to zero that induces isomorphism at the homology level. Let $V_1$ be the manifold obtained as the triple composite
\begin{equation*}
V_1=W_1\cup_{Y_2}W_2\cup_{Y_3} W_3.
\end{equation*}
We introduce a two dimensional family of (possibly degenerate) metrics parametrized by a pentagon as follows. This manifold contains five separating hypersurfaces, namely $Y_2$, $Y_3$, the two three spheres $S_1$ and $S_2$ and the manifold $R_1$ homeomorphic to $S^1\times S^2$ which is the boundary of a regular neighborhood of the $(-1)$-spheres $E_1$ and $E_2$ containing both $S_1$ and $S_2$. We can arrange them cyclically as $Y_2,R_1, Y_3, S_2$ and $S_1$ so that each of them intersects only its two neighbors, see Figure \ref{triplecomp}.
\begin{figure}[here]
  \centering
\def\svgwidth{0.7\textwidth}
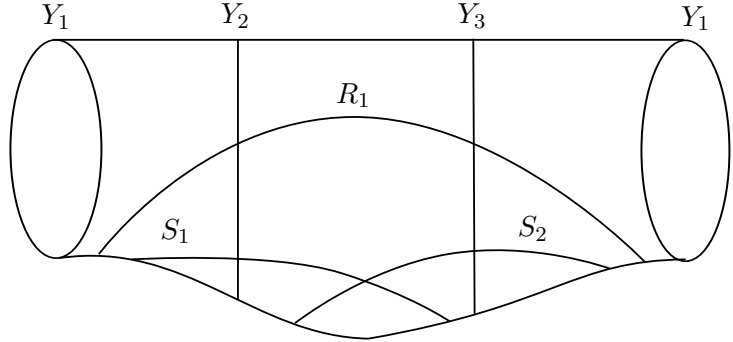
    \caption{The five hypersurfaces in the triple composite.}\label{triplecomp}
\end{figure} 
For each pair $S,S'$ of \textit{non} intersecting hypersurfaces, we define the family of metrics parametrized by $\mathbb{R}^{>0}\times\mathbb{R}^{>0}$ by inserting cylinders $[-T_S,T_S]\times S$ and $[-T_{S'},T_{S'}]\times S$. This can be completed to a family of riemannian metrics over the ``square''
\begin{equation*}
\bar{P}(S,S')\cong [0,\infty]\times [0,\infty].
\end{equation*}
The five families obtained this way fit along their five edges corresponding to families of metrics in which only one of the $T_S$ is nonzero. Hence we can obtain a family of metrics on the pentagon $\bar{P}$ obtained as their union, see Figure \ref{pentagon}. For each hypersurface $S$ there is an edge $\bar{Q}_S$ of the pentagon (consisting of two of the edges of the squares) where $T_S=\infty$. One can arrange that the family of metrics is such that $R_1,S_1$ and $S_2$ have positive scalar curvature metrics.
\begin{figure}[here]
  \centering
\def\svgwidth{0.5\textwidth}
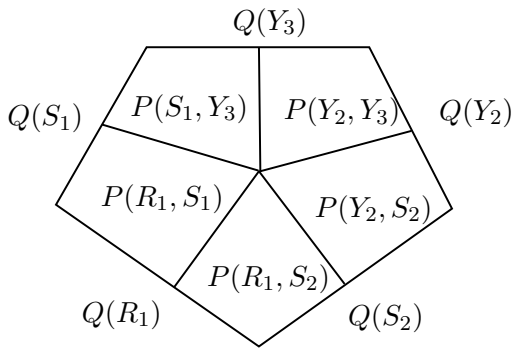
    \caption{The family of metrics $\bar{P}$.}\label{pentagon}
\end{figure} 
\par
One then considers the compactified moduli spaces of solutions parametrized by such a family, and use them to construct maps between the chain complexes. The two strata corresponding to the edges $Q(Y_2)$ and $Q(Y_3)$ are exactly those that define the maps $\check{H}_2\circ \check{f}_1$ and $\check{f}_3\circ \check{H}_1$. Notice that unlike the classical case, the sum of these two maps is not a chain map, as $\check{H}_2$ is not a chain homotopy between $\check{f}_3\circ \check{f}_2$ and zero. The maps corresponding to the edges $Q(S_1)$ and $Q(S_2)$ correspond punctured cobordisms with an additional punctured $\overline{\mathbb{C}P}^2$ component.
\par
The most interesting edge is the one given by $Q(R_1)$. The hypersurface $R_1$ is homeomorphic to $S^2\times S^1$, hence because of positive scalar curvature the only interesting spin$^c$ structure is the one with torsion first Chern class $\spin_0$. As in Section $4.4$ of \cite{Lin} we can fix a small regular perturbation with only two reducible solutions in the blow-down $\alpha_1$ and $\alpha_0$ corresponding to the spin connections $B_1$ and $B_0$ and no irreducible solutions. This is induced by a smooth $\jmath$-invariant Morse function on the torus $\mathbb{T}$ of flat connections
\begin{equation*}
f:\mathbb{T}\rightarrow\mathbb{R}
\end{equation*}
with exactly two critical points via a gauge equivariant retraction of $\mathcal{B}^{\sigma}(R_1,\spin_0)$ onto $\mathbb{T}$. We can suppose that $\alpha_1$ is the maximum, so that there are exactly two trajectories connecting $\alpha_1$ to $\alpha_0$ in the blow-down. We call the critical submanifolds in the blow-up $[\mathfrak{C}^{\mu}_{i}]$ for $\mu=0,1$ and $i\in\mathbb{Z}$. Here the superscript indicates the reducible solution on which the submanifold is lying over, and the index the eigenvalue it corresponds to. As usual, the index zero is for the first stable critical submanifold. In this case the contibutions of moduli spaces lying over the two trajectories connecting $\alpha_1$ to $\alpha_0$ in the blow-down cancel each other so that the homology is just the direct sum of the homologies of the critical submanifolds.
\par
This hypersurface $R_1$ defines a decomposition of the triple composite as the union of two four manifolds. The first, which we call $U_1$, has three boundary components and is the complement in $Y_1\times [-1,1]$ of a neighborhood of $K\times \{0\}$. The second one, which we call $N_1$, is the complement in $\overline{\mathbb{C}P}^2$ of an unknotted loop. Its second homology is generated by the two exceptional spheres $E_1$ and $E_2$, and spin$^c$ structures that restricts to $\spin_0$ on the boundary is uniquely determined by
\begin{equation*}
\langle c_1(\spin_k),[E_1]\rangle=\langle c_1(\spin_k),[E_2]\rangle=2k-1.
\end{equation*}
On the manifold with one cylindrical end $(N_1)^*$ we can consider the moduli spaces $M_k(N_1^*, [\mathfrak{C}_i^{\mu}])_{\bar{Q}}$ where $\bar{Q}$ is $\bar{Q}(R_1)$ relative to the spin$^c$ structure $\spin_k$ of solutions converging to $[\mathfrak{C}^{\mu}_i]$. We have the following lemma (see Lemma $5.7$ and following corollaries in \cite{KMOS}).
\begin{lemma}\label{dimension}
The dimension of the moduli space $M_k(N_1^*, [\mathfrak{C}_i^{\mu}])_{\bar{Q}}$ is given by
\begin{equation*}
\mathrm{dim} M_k(N_1^*, [\mathfrak{C}_i^{\mu}])_{\bar{Q}}=
\begin{cases}
-\mu-k(k-1)-4i, & i\geq 0\\
-\mu-k(k-1)-4i-1 & i<0.
\end{cases}
\end{equation*}
In particular the moduli spaces $M^{\mathrm{red}}_k(N_1^*, [\mathfrak{C}_i^{\mu}])_{\bar{Q}}$ are empty for all $i\geq 0$.
\end{lemma}

The moduli spaces above define chains in the critical submanifolds of $R_1$. In particular, we can consider $\sigma^s$, the one in the stable critical manifolds, $\sigma^u_0$ in the unstable critical manifolds above $\alpha_0$ and $\sigma^u_1$ in the critical manifolds above $\alpha_1$. We use the chains $\sigma^u_{\mu}$ in the unstable critical submanifolds to define the linear maps
\begin{equation*}
L^{o}_{o,\mu}: C^o(Y_1)\rightarrow C^o(Y_1)
\end{equation*}
and its seven companions as the map induced by the manifold $U_1$ with three ends by fibered product with $\sigma^u_{\mu}$ on the $R_1$ end. We denote the sum of the two by dropping the $\mu$ and we can combine then in the maps
\begin{align*}
\check{L}_1&=\check{L}_{1,0}+\check{L}_{1,1}: \check{C}(Y_1)\rightarrow \check{C}(Y_1)\\
\check{L}_1&=\begin{bmatrix}
L^o_o & L^u_o\bar{\partial}^s_u+\partial^u_o\bar{L}^s_u\\
L^o_s & \bar{L}^s_s+L^u_s\bar{\partial}^s_u+\partial^u_s\bar{L}^s_u.
\end{bmatrix}
\end{align*}
Similarly, one can define the map
\begin{equation*}
G^o_o: C^o(Y_1)\rightarrow C^o(Y_1)
\end{equation*}
and its seven companions induced by the fiber products with the moduli spaces on the triple composite parametrized by the pentagon of metrics $\bar{P}$, and the maps
\begin{align*}
\bar{r}^s_u: C^s(Y_1)\rightarrow C^u(Y_1)\\
\bar{r}^s_s: C^s(Y_1)\rightarrow C^s(Y_1)
\end{align*}
obtained by fiber products of the moduli spaces on the manifold with three ends $U_1^*$ with the chain $\sigma^s$ in the critical stable manifolds of $R_1$.
Finally we define
\begin{align*}
\check{G}_1&: \check{C}_{\bullet}(Y_1)\rightarrow \check{C}_{\bullet}(Y_1)\\
\check{G}_1&=\begin{bmatrix}
a & b \\ c & d
\end{bmatrix}
\end{align*}
where
\begin{align*}
a&= G^o_o\\
b&=\partial^u_o\bar{G}^s_u+G^u_o\bar{\partial}^s_u+m^u_o\bar{H}^s_u+H^u_o\bar{m}^s_u+\partial^u_o\bar{r}^s_u\\
c&=G^o_s\\
d&=\bar{G}^s_s+\partial^u_s\bar{G}^s_u+G^u_s\bar{\partial}^s_u+m^u_s\bar{H}^s_u+H^u_s\bar{m}^s_u+\partial^u_s\bar{r}^s_u+\bar{r}^s_s.
\end{align*}
Here again the $m^*_*$ are the components of the corresponding maps $\check{f}_i$. The following is the analogue of Proposition $5.5$ in \cite{KM}. We rephrase it in an alternative way because in our case the map $\check{f}_3\circ \check{H}_1+\check{H}_2\circ\check{f}_1$ is not a chain map.
\begin{lemma}\label{homotopyG}
The map $\check{G}_1$ is a chain homotopy between the map arising as the sum of $\check{L}_1$ and the maps induced by the moduli spaces parametrized by $Q(S_1)$ and $Q(S_2)$ and the map $\check{f}_3\circ \check{H}_1+\check{H}_2\circ\check{f}_1$.
\end{lemma}
\begin{proof}
The proof follows as usual by identifying the codimension one strata of the moduli spaces involved in the definition of $\check{G}_1$, which are described by the same formulas as in \cite{KMOS}. In particular the last map in the statement corresponds to the edges of the pentagon $Q(Y_2)$ and $Q(Y_3)$.
\end{proof}

\vspace{0.5cm}
We have the following key result, which is the analogue in our case of Lemma $5.10$ in \cite{KMOS}. Because our critical submanifolds are two dimensional, for our purposes we are only interested in the moduli spaces which have dimension at most two, as the others are degenerate as discussed in Section \ref{review}.
\begin{lemma}\label{mainchain}
Suppose $k\geq 1$. Then the chain $\sigma^u_1$ consists of a generator of the top homology of the critical submanifold $[\mathfrak{C}^1_i]$ for $i=-k(k-1)/4-1$ for $k$ congruent to $0,1$ modulo four, while for $k=2,3$ modulo four it consists of an even number of points in $[\mathfrak{C}^1_i]$ for $i=-k(k-1)/4-1/2$.
\end{lemma}

\begin{proof}
We first recall a result on the unperturbed anti-self-duality equations
\begin{equation*}
F_{A^t}^+=0
\end{equation*}
on the manifold $N_1^*$ from the proof of Lemma $5.10$ in \cite{KMOS}. Given a metric $g$ on $N_{1}^*$ which is standard on the end there is a unique spin$^c$ solution to such equation $A(k,g)$ with $L^2$ curvature. This is because the manifold has no first homology and no self-dual, square integrable harmonic two forms (as the image of the relative second homology in the absolute one is zero). On the cylindrical end this connection $A(k,g)$ it is asymptotically flat so it defines a point
\begin{equation*}
\theta_k(g)\in\mathbf{S}
\end{equation*}
where $\mathbf{S}$ is the circle of flat spin$^c$ connections on $S^2\times S^1$. On the family of metrics $\bar{Q}(R_1)=[-\infty,\infty]$ we have that $\theta_k(\pm\infty)$ is a spin connection. For example at $-\infty$ the manifold decomposes in two pieces, one of which is a punctured $S^2\times D^2$ with cylindrical ends, so it carries no $L^2$ harmonic forms. The key point in the proof of Lemma $5.10$ in \cite{KMOS} is to show that the connection at the two ends of this family differ, so that without loss of generality we can assume $\theta_k(-\infty)=\alpha_1$ and $\theta_k(+\infty)=\alpha_0$. Furthermore map $\theta_{1-k}(g)$ is obtained by conjugation on the circle.
\par
We then consider the moduli spaces with asymptotics into $[\mathfrak{C}^1_i]$, which for dimensional reasons is interesting when it has dimension two or zero. In the first case, $k=0,1$ modulo $4$ and and $i=k(k-1)/4-1$. We claim that for some choice of perturbations:
\begin{itemize}
\item the stratum $M_k(N_1^*, [\mathfrak{C}_i^{1}])_{-\infty}$ is a generator of the one dimensional $\jmath$-invariant homology of the critical submanifold $[\mathfrak{C}_i^{1}]$;
\item the stratum $M_k(N_1^*, [\mathfrak{C}_i^{1}])_{+\infty}$ is empty.
\end{itemize}
As we are only dealing with reducible solutions, the first claim follows in the same way as in Proposition \ref{modulib2p}. Indeed, our cobordism has $b_1=0$ and before adding the extra perturbation in the blow up of the configuration space of the cobodism the stratum over $-\infty$ consists of a two dimensional sphere of reducibles lying over the spin connection. As in the case considered in Proposition \ref{modulib2p}, it is obstructed in codimension one, and the same construction of the additional perturbation carries over. Furthermore as in that setting the moduli spaces are already compact before we compactify them because there are no possible breaking points, as there are no trajectories from $\alpha_0$ to $\alpha_1$. The gluing results regarding our moduli spaces then imply that the union of the moduli spaces
\begin{equation*}
M_k(N_1^*, [\mathfrak{C}_i^{\mu}])_{\bar{Q}}\cup M_{1-k}(N_1^*, [\mathfrak{C}_i^{\mu}])_{\bar{Q}}
\end{equation*}
is a $\jmath$-equivariant generator of the top homology of the critical submanifold, hence the claim. The second claim is clear as $\theta_k(+\infty)$ is $\alpha_0$.
\par
Finally when $k=2,3$ modulo four the strata $M_k(N_1^*, [\mathfrak{C}_i^{1}])_{\pm\infty}$ are both empty by transversality, hence the moduli spaces consist of an even number of points by symmetry.
\end{proof}

\vspace{0.5cm}
Before proving Proposition \ref{mainiso}, we need to discuss the $\Rin$-module structure on the mapping cones we have defined.
\begin{lemma}\label{modulestr}
The mapping cone $H_*(M_{\check{f}_1})$ of the chain map $\check{f}_1$ is an $\Rin$-module, and the mapping cone triangle for $\HSt_{\bullet}(W_1)$ is an exact triangle of $\Rin$-modules. The coboundary map $\delta$ in Proposition \ref{mainiso} is a map of $\Rin$-modules.
\end{lemma}
\begin{proof}
The construction of the module structure follows the ideas in \cite{BloS}, and is described in Figure \ref{conemodule}.
\begin{figure}[here]
  \centering
\def\svgwidth{0.5\textwidth}
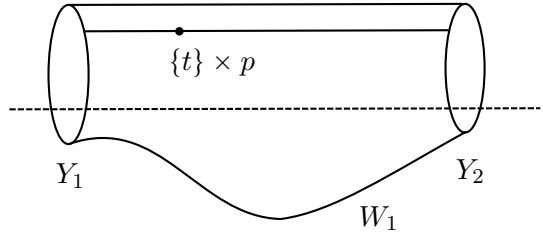
    \caption{The module structure on the mapping cone. The cobordism is a product $(Y\setminus\mathrm{nbhd}(K))\times [-1,1]$ above the dashed line.}\label{conemodule}
\end{figure} 
Consider a point $p$ which is not in the neighborhood of the knot $K$ where the surgery operation is performed. In particular we can consider $p$ as a point in both $Y_1$ and $Y_2$, and use the same ball $B_p$ embedded in $\mathbb{R}\times Y_i$ centered at $p$ to compute the map induced (for example) by $V\in\Rin$ by looking at the moduli spaces that are asymptotic to the second unstable critical submanifold $[\Cr_{-2}]$ on the additional incoming end. Here we assume that the metric has positive scalar curvature on the boundary of $B_p$ and is a product near there. Call the associated chain maps $\check{V}_1$ and $\check{V}_2$. There is a natural chain homotopy $\check{\mathcal{H}}_1$ between $\check{f}_1\circ \check{V}_1$ and $\check{V}_2\circ \check{f}_1$ obtained by considering the compactifications of the moduli spaces of trajectories parametrized by the moving point $\{t\}\times p$. Indeed, we can identify the subset
\begin{equation*}
\mathbb{R}\times(Y\setminus \mathrm{nbhd}(K))\subset W_1^*
\end{equation*}
and consider the union of the moduli spaces of trajectories on the cobordism with a puncture at $\{t\}\times p$ that are asymptotic to $[\Cr_{-2}]$ on this additional end. As usual, we can compactify these moduli spaces and use them to define a map $\mathcal{H}_1$ satisfying
\begin{equation*}
\check{\partial}_2\circ \check{\mathcal{H}}_1+\check{\mathcal{H}}_1\circ\check{\partial}_1=\check{f}_1\circ \check{V}_1+\check{V}_2\circ \check{f}_1.
\end{equation*}
The endomorphism of $M_{\check{f}_1}$ defined by the matrix
\begin{equation*}
\begin{pmatrix}
\check{V}_2 & \check{\mathcal{H}}_1 \\
0 & \check{V}_1
\end{pmatrix}
\end{equation*}
is then a chain map, and we define the induced map to be the action of $V$ of $H_*({M}_{\check{f}_1})$. The usual arguments show that this is well defined, and the maps in the triangle commute with this map.
\par
The module structure on the iterated mapping cone $H_*(\check{C})$ is defined in an analogous way. Indeed the same construction above applied to the cobordism $W_2$ leads to a chain homotopy $\check{\mathcal{H}}_2$ such that
\begin{equation*}
\check{\partial}_3\circ \check{\mathcal{H}}_2+\check{\mathcal{H}}_2\circ\check{\partial}_2=\check{f}_2\circ \check{V}_2+\check{V}_3\circ \check{f}_1,
\end{equation*}
where $\check{V}_3$ is the analogous chain map inducing the action of $V$ on $\HSt_{\bullet}(Y_3)$. We claim that there is a map $\check{\mathcal{G}}_1$ from $\check{C}_1$ to $\check{C}_3$ with the property that
\begin{equation}\label{finalh}
\check{\partial}_3\circ\check{\mathcal{G}}_1+\check{\mathcal{G}}_1\circ\check{\partial}_3=\check{f}_2\circ\check{\mathcal{H}}_1+ \check{\mathcal{H}}_2\circ\check{f}_1+\check{H}_1\circ \check{V}_1+\check{V}_3\circ \check{H}_1,
\end{equation}
so that
\begin{equation*}
\begin{pmatrix}
\check{V}_3 & \check{\mathcal{H}}_2 & \check{\mathcal{G}}_1 \\
0 & \check{V}_2 & \check{\mathcal{H}}_1\\
0 & 0 & \check{V}_1
\end{pmatrix}
\end{equation*}
is a chain map. We use this to map define the action of $V$ on $H_*(\check{C})$. The map $\check{\mathcal{G}}_1$ is constructed as follows. The cobordism with cylindrical ends attached $(W_2\circ W_1)^*$ contains a copy of $(Y\setminus \mathrm{nbhd}(K))\times\mathbb{R}$ hence in particular the line $p\times \mathbb{R}$, so we can consider the moduli space parametrized by $\mathbb{R}\times \mathbb{R}$ where the first component parametrizes the position of the point while the second parametrizes the family of metrics used to define the chain homotopy $\check{H}_1$. The map $\check{\mathcal{G}}_1$ is then defined by considering the moduli spaces of solutions parametrized by this family, and the identity (\ref{finalh}) follows as usual by identifying the contributions of the four edges of the square $[-\infty,\infty]\times [-\infty,\infty]$.
\end{proof}

\vspace{0.5cm}
\begin{proof}[Proof of Proposition \ref{mainiso}]
Consider the map $\check{G}$ on the iterated mapping cone $\check{C}$ given by
\begin{equation*}
\check{G}=
\begin{pmatrix}
\check{G}_3 & 0 & 0 \\
\check{H}_3 & \check{G}_2 & 0\\
\check{f}_3 & \check{H}_2 & \check{G}_1
\end{pmatrix}.
\end{equation*}
Here $\check{G}_2$ and $\check{G}_3$ are the maps induced by the moduli spaces parametrized by the pentagon of metrics and perturbations on the other two triple composites. They satisfy the properties of $\check{G}_1$ that we have discussed above.
We define the chain map
\begin{equation*}
\check{\varphi}: \check{C}\rightarrow \check{C}
\end{equation*}
given by
\begin{equation*}
\check{\varphi}= \check{\partial}\circ \check{G}+\check{G}\circ\check{\partial}.
\end{equation*}
The map $\check{\varphi}$ is nullhomotopic by definition, and our claim is that it also induces an isomorphism at the homology level. We have (for example) the identity
\begin{equation}\label{bigmodulez}
\check{V}\circ\check{\varphi}+\check{\varphi}\circ\check{V}=\check{\partial}\circ(\check{G}\circ\check{V}+\check{V}\circ\check{G})+(\check{G}\circ\check{V}+\check{V}\circ\check{G})\circ\check{\partial}
\end{equation}
so that the map induced in homology is a map of $\Rin$-modules. Using the relations of Lemma \ref{homotopyG} we can write the map $\check{\varphi}$ as the matrix
\begin{equation*}\label{bigmap1}
\check{\varphi}=
\begin{pmatrix}
\check{L}_3+\check{h}_3& \check{f}_2\check{G}_2+\check{G}_3\check{f}_2+\check{H}_1\check{H}_2 & \check{H}_1\check{G_1}+\check{G}_3\check{H}_1 \\
\check{g}_3 & \check{L}_2+\check{h}_2 & \check{f}_1\check{G}_1+\check{G}_2\check{f}_1+\check{H}_3\check{H}_1\\
0 & \check{g}_2 &\check{L}_1+\check{h}_1
\end{pmatrix}.
\end{equation*}
The lower diagonal terms $\check{g}_i$ are those appearing in Lemma \ref{homotopies}, and the maps $\check{h}_i$ indicate the maps induced by the two edges in the pentagon corresponding to the blowups. For example, using the notation we have adopted throughout the section, the map $\check{h}_1$ is defined using the moduli spaces parametrized by the families of metrics $Q(S_1)$ and $Q(S_2)$. Notice that these are \textit{not} chain maps. Unfortunately there is not a natural filtration respected by this map $\check{\varphi}$. Our strategy is to show that its mapping cone (which has a natural filtration) has trivial homology. In particular, we consider the chain complex whose underlying vector space is the sum of two copies of $\check{C}$ (where we distinguish the elements and groups in the first copy with the apostrophe)
\begin{equation*}
\tilde{C}=\check{C}'\oplus \check{C}
\end{equation*}
and differential
\begin{equation*}
\tilde{\partial}=
\begin{pmatrix}
\check{\partial}'_3 & \check{f}'_2 & \check{H}'_1 & \check{L}_3+\check{h}_3 & \ast & \ast \\
0 & \check{\partial}'_2 & \check{f}'_1 & \check{g}_3 & \check{L}_2+\check{h}_2 & \ast \\
0 & 0 & \check{\partial}_1 & 0 & \check{g}_2 & \check{L}_1+\check{h}_1\\
0 & 0 &0 & \check{\partial}_3 & \check{f}_2 & \check{H}_1\\
0 & 0 & 0 &0 & \check{\partial}_2 & \check{f}_1\\
0 & 0 &0&0 & 0 & \check{\partial}_1
\end{pmatrix}.
\end{equation*}
This chain complex has a natural filtration induced by the upper triangular structure of the differential $\tilde{\partial}$. Because the left lower entry of $\check{\varphi}$ vanishes, the $E^1$ page of the associated spectral sequence does not involve differentials between the corresponding subquotients of $\check{C}$ and $\check{C}'$. In particular, the $E^2$ page of the spectral sequence is given by
\begin{center}
\begin{tikzpicture}
\matrix (m) [matrix of math nodes,row sep=1.5em,column sep=2.5em,minimum width=2em]
  {
  \mathrm{coker}\check{f}'_2 & \mathrm{ker}\check{f}'_2/\mathrm{Im}\check{f}'_1 & \mathrm{ker}\check{f}'_1\\
   \mathrm{coker}\check{f}_2 & \mathrm{ker}\check{f}_2/\mathrm{Im}\check{f}_1 & \mathrm{ker}\check{f}_1\\};
  
  \path[-stealth]
  (m-2-1) edge node [above]{$\check{g}_3$} (m-1-2)
  (m-2-2) edge node [above]{$\check{g}_2$} (m-1-3)
  ;
  
  \draw(m-1-3) edge[out=160, in=20, ->] node [above]{$\check{H}'_1$}  (m-1-1);
  \draw(m-2-3) edge[out=200, in=340, ->] node [below]{$\check{H}_1$}  (m-2-1);
\end{tikzpicture}
\end{center}
where by an abuse of the notation we are considering the maps induced on the subquotients by the indicated maps. On the other hand Lemma \ref{homotopies} tells us that the chain maps $\check{g}_2$ and $\check{g}_3$ induce at the homology level the maps $\check{f}_3\circ \check{f}_2$ and $\check{f}_1\circ \check{f}_3$. In particular $\check{g}_2$ is zero on the kernel of $\check{f}_2$ and the image of $\check{g}_3$ is contained in the image of $\check{f_1}$, so the diagonal maps in the diagram above are both zero. So the subquotients of the two chain complexes $\check{C}$ and $\check{C}'$ do not interact at this page either, and the $E^3$ page is simply
\begin{center}
\begin{tikzpicture}
\matrix (m) [matrix of math nodes,row sep=2em,column sep=2em,minimum width=2em]
  {\mathrm{coker}\check{f}'_2/\mathrm{Im}\check{H}'_1 & \mathrm{ker}\check{f}'_2/\mathrm{Im}\check{f}'_1
 & \mathrm{ker}\check{f}'_1\cap \mathrm{ker}\check{H}'_1\\
\mathrm{coker}\check{f}_2/\mathrm{Im}\check{H}_1 & \mathrm{ker}\check{f}_2/\mathrm{Im}\check{f}_1
 & \mathrm{ker}\check{f}_1\cap \mathrm{ker}\check{H}_1\\};

\path[-stealth]
(m-2-1) edge node [left]{$ \check{L}_3+\check{h}_3$}(m-1-1)
(m-2-2) edge node [left]{$ \check{L}_2+\check{h}_2$}(m-1-2)
(m-2-3) edge node [left]{$ \check{L}_1+\check{h}_1$}(m-1-3);
;
\end{tikzpicture}
\end{center}
Our claim is that the vertical maps are isomorphisms, so that the $E^4$ page of the spectral sequence is zero, proving our claim that $\check{\varphi}$ is an isomorphism. From equation (\ref{bigmodulez}) it follows that the objects involved $\Rin$-modules and the vertical arrows are maps of $\Rin$-module. The maps $\check{h}_i$ above involve the multiplication by the element $Q^2$ in $\Rin$, as they are defined via moduli spaces on manifolds parametrized by family of metrics on which a blow up is already stretched to infinity (as for example in Proposition \ref{blowup}).
\par
Recall from Sections $3.3$ and $4.4$ of \cite{Lin} that the group $\HSb_{\bullet}(Y)$ is naturally a module over
\begin{equation*}
\Lambda^*\left(H_1(Y;\mathbb{Z})/\mathrm{Tor}\otimes \ztwo\right)\otimes \Rin.
\end{equation*}
Indeed, consider a closed embedded loop $\gamma$ in $Y$ representing a given homology class $x$. A neighborhood of this loop has boundary $S^2\times S^1$ and we can suppose that the metric and perturbations on this are the same as we discussed above. In particular we have that
\begin{equation*}
\HSf_{\bullet}(S^2\times S^1)= (\Rin\oplus \Rin\langle-1\rangle)\langle-1\rangle,
\end{equation*}
where the first $\Rin$ summand corresponds to the critical submanifolds $[\Cr^1_i]$ while the second one to the critical submanifolds $[\Cr^1_i]$. Again here the brackets denote the grading shift. The cobordism $[-1,1]\times Y\setminus \mathrm{nbhd}(\{0\}\times \gamma)$ induces a map
\begin{equation*}
\HSt_{\bullet}(Y)\otimes  (\Rin\oplus \Rin\{-1\})\{-1\}\rightarrow \HSt_{\bullet}(Y).
\end{equation*}
When restricting to the elements of the first $\Rin$ summand, we recover the usual $\Rin$-module structure, while the elements in the second summand correspond to the action of the elements $x\otimes \Rin$. Because of Lemma \ref{mainchain}, the term $\check{L}_{i,1}$ induce the sum of a power series in $\Rin$ with leading term $1$. All the terms in the summand $\check{L}_{i,0}$ of $\check{L}_i$ also involve the multiplication by the nilpotent element
\begin{equation*}
[K]\in \Lambda^*\left(H_1(Y;\mathbb{Z})/\mathrm{Tor}\otimes \ztwo\right).
\end{equation*}
Hence each of the vertical maps on the $E^3$ page is a sum of an isomorphism $\check{L}_{i,1}$ and a nilpotent map $\check{L}_{i,0}+\check{h}_i$ which commute because of the $\Rin$-module structure, so it is an isomorphism.
\end{proof}

\vspace{0.5cm}
From Proposition \ref{mainiso} we can conclude the main result of the present paper.
\begin{proof}[Proof of Theorem \ref{exacttr}]
The existence of the triangle follows from our discussion and Lemma \ref{triangledetection}. It is a triangle of $\Rin$-modules thanks to Lemma \ref{modulestr}. It remains to show that the third map is well defined (i.e. independent of the choices we have made). The mapping cone construction satisfies the following naturality property: given two homotopic chain maps $f,f'$ between chain complexes $C$ and $C'$ there is an isomorphism between the two mapping cones such that the mapping cone exact triangles commute. In fact, if $h$ is a chain homotopy between $f$ and $f'$ the canonical isomorphism is given by the matrix
\begin{equation*}
M(h)=\begin{pmatrix}\mathrm{Id}_C & h \\ 0 & \mathrm{Id}_{C'}\end{pmatrix}.
\end{equation*}
Furthermore given another such chain homotopy $h'$, if there exists a map
\begin{equation*}
k: C\rightarrow C'
\end{equation*}
such that $\partial'\circ k+k\circ\partial=h-h'$ then the induced isomorphism is the same, as the two maps $M(h)$ and $M(h')$ are homotopic via the map $\begin{pmatrix}0 & k \\ 0 & 0\end{pmatrix}$.
\par
In our case the two maps $f$ and $f'$ correspond to two different regular choices of metric and perturbation $(g_0, \mathfrak{p}_0)$ and $(g_1, \mathfrak{p}_1)$. In order to identify the mapping cones we consider chain homotopy $h$ constructed by considering the moduli spaces parametrized by a regular path $(g_t, \mathfrak{p}_t)$ for $t\in[0,1]$ connecting these two choices. Because the space of metrics and perturbations is contractible, any two such paths are homotopic via a generic homotopy $h_{s,t}$ (which we can think as a regular two dimensional family of metrics and perturbations) relative to their endpoints. The map $k$ is then constructed by considering the moduli spaces parametrized by this two dimensional family.
\par
The analogous construction carries over to show that the iterated mapping cone (hence the boundary map $\delta$) is natural. Suppose we have two iterated mapping cones corresponding to triples $f_1,f_2, H_1$ and $f_1', f_2',H_1'$. Given for $i=1,2$ chain homotopies $h_i$ as above between $f_i$ and $f_i'$ we claim that there is a map
\begin{equation*}
K: \check{C}_1\rightarrow \check{C}_3
\end{equation*}
satisfying the identity
\begin{equation}\label{lasth}
\partial_3\circ K+K\circ \partial_1= f_2'\circ h_1+ h_2\circ f_1 +H_1+H_1',
\end{equation}
so that the map
\begin{equation*}
\begin{pmatrix}\mathrm{Id} & h_2 & K \\ 0 & \mathrm{Id} & h_1 \\ 0 & 0 & \mathrm{Id}\end{pmatrix}
\end{equation*}
is an isomorphism between the two iterated mapping cones.
\par
In our case (where we add checks to be consistent with our notation) map $\check{K}$ is constructed by considering the moduli spaces parametrized by a pentagon of metrics (which are possibly degenerate) and perturbations. The five vertices of the pentagon correspond to the maps $\check{f}_2\circ \check{f}_1$, $\check{f}_2'\circ \check{f}_1$, $\check{f}_2'\circ \check{f}_1'$ and the two endpoints $p$ and $p'$ of the homotopies $\check{H}_1$ and $\check{H}_1'$ corresponding to the blow up
\begin{equation*}
W_2\circ W_1=\overline{W_3}\# \overline{\mathbb{C}P}^2
\end{equation*}
stretched to infinity. Four of the edges correspond to the four terms in the expression of $\check{\partial}_3\circ \check{K}+\check{K}\circ \check{\partial}_1$ in equation (\ref{lasth}). The fifth edge is a path between $p$ and $p'$ throughout degenerate metrics for which the blow up is stretched to infinity. We can choose this path so that the copy of $S^3$ along which the connected sum is performed always has positive scalar curvature. The moduli spaces parametrized by this edge do not contribute to the boundary terms for the same reason why the composite map is zero, see Proposition \ref{blowup}. These five edges can be filled to a pentagon using again the contractibility of the space of metrics and perturbations, and the induced isomorphism is well defined for the same reason.
\end{proof}

\vspace{1cm}

\section{Computations from the Gysin exact sequence}\label{compgysin}

In this Section we show that when the usual monopole Floer homology of a three manifold is very simple, the $\Pin$-monopole Floer homology can be recovered in purely algebraic terms from the Gysin exact sequence
\begin{equation}\label{gysinseq}
\dots\stackrel{\cdot Q}{\longrightarrow} \HSt_k(Y)\stackrel{\iota_*}{\longrightarrow} \HMt_k(Y)\stackrel{\pi_*}{\longrightarrow} \HSt_k(Y)\stackrel{\cdot Q}{\longrightarrow}\HSt_{k-1}(Y)\stackrel{\iota_*}{\longrightarrow}\dots
\end{equation} 
For a rational number $d$ denote by $\mathcal{T}^+_d$ the graded $\ztwo[[U]]$-module $\ztwo[U^{-1},U]]/U\ztwo[[U]]$, where $1$ has degree $d$. In particular $\mathcal{T}^+_0$ is isomorphic as a graded $\ztwo[[U]]$-module to the Floer homology group $\HMt_{\bullet}(S^3)$.
\par
It is useful to have the basic example of $S^3$ in mind, which we briefly recall. We have that
\begin{equation}\label{gysinS3}
\HMt_k=
\begin{cases}
\ztwo \text{ if }k\geq 0 \text{ even}\\
0 \text{ otherwise }
\end{cases}\qquad
\HSt_k=
\begin{cases}
\ztwo \text{ if }k\geq 0, k\neq 3 \pmod{4}\\
0 \text{ otherwise }
\end{cases}
\end{equation}
and the relevant Gysin exact sequence has the form
\begin{equation}
\xymatrix{
\ztwo_{4n+2} \ar[r] & \ztwo_{4n+2} & \ztwo_{4n+2}\ar[lld]\\
\ztwo_{4n+1} & 0 & \ztwo_{4n+1}\ar[lld] \\
\ztwo_{4n} & \ztwo_{4n} \ar[r] & \ztwo_{4n}
}
\end{equation}
for every $n\geq 0$. Here the indices denote the gradings, while the arrows are for the maps which are not trivial. Notice that these can be deduced directly by the group structure of the two groups and the exactness of the sequence. For a given rational homology sphere $Y$ and a self conjugate spin$^c$ structure $\spin$, we know that the Gysin sequence for $(Y,\spin)$ looks like the one of equation (\ref{gysinS3}) up to grading shift in degrees high enough. This motivates the following definition.
\begin{defn}
An \textit{abstract Gysin sequence} $\mathcal{G}$ consists of the following data:
\begin{itemize}
\item a $\ztwo[[U]]$-module $M$ and a $\Rin$-module $S$, both graded by a coset of $\mathbb{Z}$ in $\mathbb{Q}$ and bounded below;
\item an exact triangle of $\Rin$-modules
\begin{center}
\begin{tikzpicture}
\matrix (m) [matrix of math nodes,row sep=2em,column sep=2em,minimum width=2em]
  {
  S && S\\
  &M &\\};
  \path[-stealth]
  (m-1-1) edge node [above]{$e_*$} (m-1-3)
  (m-2-2) edge node [left]{$\pi_*$} (m-1-1)
  (m-1-3) edge node [right]{$\iota_*$} (m-2-2)  
  ;
\end{tikzpicture}
\end{center}
where the $\Rin$-module structure on $M$ is given by $Q$ acting trivially and $V$ acting as $U^2$;
\item the maps $\pi_*$ and $\iota_*$ have degree zero while $e_*$ has degree $-1$, and the triangle is isomorphic to the exact triangle (\ref{gysinS3}) in degree high enough.
\end{itemize}
\end{defn}
In particular, the module $M$ has a unique infinite dimensional $\ztwo[[U]]$-submodule, and we denote by $2h(M)$ the minimum degree in which it is non trivial. Because the exact triangle looks like (\ref{gysinS3}) in degrees high enough, we have that the group $S$ is trivial degrees $2h(M)+4N-1$ or $2h(M)+4N+1$ for $N$ big enough. Of course only one of the two possibilities is allowed.
\begin{defn}
If the group $S$ is trivial in degrees $2h(M)+4N-1$ for $N$ big enough we say that the abstract Gysin sequence $\mathcal{G}$ is \textit{even}, and we say that it is \textit{odd} otherwise.
\end{defn}
\vspace{0.5cm}
We are ready to state the main result of the present section.
\begin{prop}\label{gysin}
Suppose we are given a $\ztwo[[U]]$-module $M$ of the form
\begin{equation*}
\T^+_{2k}\oplus \ztwo^n\langle2k-1\rangle.
\end{equation*}
Then there exists a unique (up to isomorphism) abstract Gysin sequence in which $M$ fits. If $n=2m$ is even then the sequence $\mathcal{G}$ is even and
\begin{equation*}
S\cong \mathcal{S}^+_{k,k,k}\oplus \ztwo^m\langle2k-1\rangle,
\end{equation*}
while if $n=2m+1$ is odd the sequence $\mathcal{G}$ is odd and
\begin{equation*}
S\cong \mathcal{S}^+_{k+1,k-1,k-1}\oplus \ztwo^{m+1}\langle2k-1\rangle.
\end{equation*}
Suppose we are given a $\ztwo[[U]]$-module of the form
\begin{equation*}
\T^+_{2k}\oplus \ztwo^n\langle2k\rangle.
\end{equation*}
Then there exists a unique up to isomorphism abstract Gysin sequence in which $M$ fits. If $n=2m+1$ is odd then the sequence $\mathcal{G}$ is even and
\begin{equation*}
S\cong \mathcal{S}^+_{k,k,k}\oplus \ztwo^{m+1}\langle2k\rangle,
\end{equation*}
while if $n=2m$ is even the sequence $\mathcal{G}$ is odd and
\begin{equation*}
S\cong \mathcal{S}^+_{k+1,k+1,k-1}\oplus \ztwo^{m+1}\langle2k\rangle.
\end{equation*}
\end{prop}

\vspace{0.8cm}
The key idea is the following easy observation which readily follows from the exactness of the Gysin exact sequence.
\begin{lemma}\label{increase}
Given a Gysin exact sequence $\mathcal{G}$, suppose that for some $k$ we have
\begin{equation*}
M_{k-1}=0\qquad M_k=\ztwo\qquad M_{k+1}=0.
\end{equation*}
Then we have the two possibilities
\begin{equation}
\xymatrix{
\ztwo^{a+1}& 0 & \ztwo^{a+1}\ar[lld]\\
\ztwo^{a+1} & \ztwo_k\ar[r] & \ztwo^{a+1}\ar[lld] \\
\ztwo^{a} & 0 & \ztwo^{a}
}
\qquad\qquad
\xymatrix{
\ztwo^{a-1}& 0 & \ztwo^{a-1}\ar[lld]\\
\ztwo^{a} \ar[r]& \ztwo_k& \ztwo^{a}\ar[lld] \\
\ztwo^{a} & 0 & \ztwo^{a}
}
\end{equation}
where $a$ is non negative on the left case and positive on the right case.
\end{lemma}
\begin{defn}
In the first situation we say that the Gysin sequence at the level $k$ is \textit{increasing}, while in the second situation we say that the Gysin sequence is \textit{decreasing}.
\end{defn}

\begin{proof}[Proof of Proposition \ref{gysin}]
In the proof we can assume without loss of generality that $k=0$ after a grading shift. Suppose first that that we are in the last case, so that $M_0=\ztwo^{2m+1}$. Clearly the suggested $S$ fits in an abstract Gysin exact sequence, hence we need to prove uniqueness. Dimensional considerations and the exactness of the Gysin triangle imply that $S$ is trivial in negative degrees while $\pi_*: M_0\rightarrow S_0$ is surjective. Let the rank of the latter be $m+1+l$ for some $0\leq l\leq m$. To determine the structure of the whole group we can then use Lemma \ref{increase}. In particular, because the rank of $S_1$ is odd the sequence implies that the rank of $S_{3+4N}$ will be even for all $N\geq 0$, hence the Gysin sequence is even.
\par
We claim that when we look at the Gysin triangle in degrees $4N\leq i\leq 4N+3$ contains a copy of the sequence (\ref{gysinS3}) given by the image under a suitably high power of the map $V$. We prove this by induction, as it is of course true for $N$ big enough by assumption. Denote by
\begin{equation*}
v^{-N},qv^{-N},q^2v^{-N}\text{ and }u^{-2N-1}, u^{-2N}
\end{equation*}
the respective generators of this copy. Of course
\begin{equation*}
V\cdot u^{-2N-1}=u^{-2N+1},
\end{equation*}
so as $\iota_*$ maps $v^{-N}$ to $u^{-2N-1}$, $V\cdot v^{-N}$ is not zero. Denote this element by $v^{-N+1}$. This is mapped by $\iota_*$ to $u^{-2N+1}$, so this map is not zero and the sequence is decreasing at the level $4N-1$. 
This implies that $Q\cdot v^{-N+1}$ is not zero, and we call this element $qv^{-N+1}$. Similarly, we denote the image of this element under the action of $Q$ by $q^2v^{-N+1}$. This is not zero because $M_{4N+1}$ is trivial. Also the module structure implies that $\pi_*(u^{-2N+2})=q^2v^{-N+1}$.
\par
This final observation implies that at each level $4N$ for $N\geq 1$ the sequence is increasing, so the ranks of $S_{4N-1}$ form a non decreasing sequence. As it has to be zero for $N$ big enough, all these ranks are zero. By dimensional considerations, $l$ has to be zero and the result follows.
\par
The proof in the other three cases is analogous. The only difference in the odd case is that one shows that the Gysin triangle in degrees $4N+2\leq i\leq 4N+5$ for $N\geq 0$ contains a copy of the standard one.
\end{proof}

\begin{proof}[Proof of Theorem \ref{trefoil}]
This follows readily by applying Proposition \ref{gysin} to the case of the Brieskorn spheres $\Sigma(2,3,6n\pm1)$. In particular we have
\begin{align*}
\HMt_{\bullet}(-\Sigma(2,3,12k+5))&=\mathcal{T}^+_{-2}\oplus \ztwo^{2k}\langle-2\rangle \\
\HMt_{\bullet}(\Sigma(2,3,12k+1))&=\mathcal{T}^+_{0}\oplus \ztwo^{2k}\langle-1\rangle \\
\HMt_{\bullet}(-\Sigma(2,3,12k-1))&=\mathcal{T}^+_{-2}\oplus \ztwo^{2k-1}\langle-2\rangle \\
\HMt_{\bullet}(\Sigma(2,3,12k-5))&=\mathcal{T}^+_{0}\oplus \ztwo^{2k-1}\langle-1\rangle,
\end{align*}
and the result for the given orientations follows from Poincar\'e duality.
\end{proof}

\begin{remark}
It is not surprising that in general the $\Pin$-monopole Floer homology cannot be recovered from the usual counterpart, as some ambiguities may arise. For example, if $Y$ is the Brieskorn sphere $-\Sigma(2,3,11)$, we have
\begin{equation*}
\HMt_{\bullet}(Y\#Y)\cong \T^+_{-4}\oplus \ztwo^3\langle-4\rangle\oplus \ztwo\langle-3\rangle
\end{equation*}
as it can be computed by the connected sum formula in Heegaard Floer homology (see \cite{OS2}). Then the two $\Rin$-modules
\begin{equation*}
\mathcal{S}^+_{0,-2,-2}\oplus \ztwo^2\langle-4\rangle\text{  and  } \mathcal{S}^+_{-2,-2,-2}\oplus \ztwo^2\langle-4\rangle\oplus \ztwo\langle-3\rangle
\end{equation*}
both fit in an abstract Gysin sequence. Notice that in this case we cannot recover Manolescu's correction terms either. 
\end{remark}

\vspace{1cm}
\section{Examples}\label{computations}
In this section we discuss some simple computations of the $\Pin$-monopole Floer homology groups that can be done by applying the surgery exact triangle of Theorem \ref{exacttr}. In order to get acquainted with the ideas, we first start with two examples of $\infty,0, 1$ surgery on a knot in $S^3$ in which we already know all the groups involved in the computations. In general, we will label the maps by the surgery coefficient of the manifold corresponding to the domain. In this particular case the non spin cobordism is the one from $Y_1$ to $Y_{\infty}$, so the map provided by Theorem \ref{exacttr} is $\bar{F}_1$.
\vspace{0.5cm}
\begin{example}Suppose $K$ is the unknot. Then
\begin{equation*}
Y_0=S^2\times S^1\text{ and }Y_1=S^3.
\end{equation*}
We know from the discussion in the previous section (see also Section $4.4$ of \cite{Lin} for more details) that for the unique self-conjugate spin$^c$ structure $\spin_0$ we have the isomorphisms of graded $\Rin$-modules
\begin{align*}
\HSb_{\bullet}(S^2\times S^1,\spin_0)&\cong \mathcal{S}\otimes (\ztwo\oplus\ztwo\langle -1\rangle)\\
\HSt_{\bullet}(S^2\times S^1,\spin_0)&\cong \mathcal{S}^+_{0,0,0}\otimes (\ztwo\oplus\ztwo\langle -1\rangle)
\end{align*}
In this case for all the three manifolds involved the map $i_*$ is surjective. In particular, all the triangles are determined by the one for the bar version, and will focus on the latter. The map
\begin{equation*}
\HSb_{\bullet}(W_{\infty}):\HSb_{\bullet}(S^3)\rightarrow \HSb_{\bullet}(S^2\times S^1,\spin_0),
\end{equation*}
which has degree $-1$, is an isomorphism onto the summand $\mathcal{S}\langle-1\rangle$ with lower degree while the map, while the map
\begin{equation*}
\HSb_{\bullet}(W_{0}):\HSb_{\bullet}(S^2\times S^1,\spin_0)\rightarrow \HSb_{\bullet}(S^3)
\end{equation*}
which has degree zero sends the lower summand $\mathcal{S}\langle -1\rangle$ to zero and is an isomorphism when restricted to the top summand $\mathcal{S}$. Indeed the spin connection on $W_{\infty}$ restricts the spin connection $B_0$ on $S^2\times S^1$ which is the minimum, while the spin the spin connection on $W_{0}$ restricts to the spin connection $B_1$. The corresponding moduli spaces of solutions are all copies of $\mathbb{C}P^1$ on which the evaluation maps diffeomorphically. This description can also be derived from the exact triangle for the usual monopole groups. It follows that the third map (which is called $\bar{F}_1$ in our case)
\begin{equation*}
\bar{F}_1: \HSb_{\bullet}(S^3)\rightarrow \HSb_{\bullet}(S^3)
\end{equation*}
is zero. In particular it is different from the map induced by the corresponding cobordism. In fact, the cobordism $W_1$ is a twice punctured $\overline{\mathbb{C}P}^2$ so the induced map is the multiplication by a non zero power series in which each term involves the multiplication by $Q^2$.
\end{example}

\vspace{0.5cm}
\begin{example}\label{poincare}
Let $K$ be the right handed trefoil. Then $(+1)$-surgery is the Poincar\'e homology (oriented as the boundary of the negative definite $E_8$ plumbing) and we computed in \cite{Lin} that as $\Rin$-modules we have
\begin{equation*}
\HSt_{\bullet}(Y_1)\cong \mathcal{S}^+_{-1,-1,-1}
\end{equation*}
and the map $i_*$ is surjective. Similarly, the $0$-surgery is a flat torus bundle over the circle and we showed in Section $4.4$ of \cite{Lin} that as $\Rin$-modules
\begin{align*}
\HSb_{\bullet}(S^3_0(K),\spin_0)&\cong (\V_1 \oplus \V_0)\oplus (\V_{-1}\oplus \V_{-2})\\ 
\HSt_{\bullet}(S^3_0(K),\spin_0)&\cong (\V^+_1\oplus \V^+_0)\oplus(\V^+_{-1}\oplus \V^+_{-2}).
\end{align*}
In both cases the action of $Q$ (which has degree $-1$) is an isomorphism from the first summand to the second summand and from the third summand to the forth summand. It is interesting to notice that (unlike in usual monopole Floer homology) the bar group is significantly different from the case of $S^2\times S^1$. This is another manifestation of the modulo four periodicity of the groups, and is related to the fact that the trefoil knot has Arf invariant $1$. More in detail, in the blow down the situation is analogous to that of $S^2\times S^1$, with two critical points $\alpha_1$ and $\alpha_0$ connected by two trajectories related by the action of $\jmath$. The key difference is that the family of Dirac operators has spectral flow $+1$ (so in particular odd) along these paths, and in particular that the reducible solutions lying over $\alpha_0$ are shifted up in degree by $2$. This implies that the generator of the top homology of $[\Cr^1_i]$ cancels with the generator of the bottom homology of $[\Cr^0_i]$. For the remaining spin$^c$ structures the groups is zero by the adjunction inequality. Also in this case the $i_*$ maps are surjective so it suffices to determine the reducible solutions and the bar version of the groups.
\par
The maps $\HSb_{\bullet}(W_{\infty})$ and $\HSb_{\bullet}(W_0)$ have degree respectively $-1$ and $0$, and the first map is given by multiplication by $Q$ onto the first tower, while the latter zero on the first tower and the identity on the second one. This follows from the same discussion of the moduli spaces on the cobordisms $W_{\infty}$ and $W_0$ as above, the only difference being the cancellations happening in the chain complex of $Y_0$.
\par
In particular the map $\bar{F}_1$ is non-zero in this case. It is not straightforward identify the map provided by the theorem in this case. Nevertheless we can say that the topmost homogeneous part of $\bar{F}_1$ lies in degree zero and that the latter is an isomorphism in degree divisible by four and zero otherwise. This statement follows from the degrees of the \textit{to} groups involved in the triangle and the module structure. Indeed the generator of $\HSt_0(Y_1)$ has to be mapped to the generator of $\HSt_0(S^3)$ for degree reasons, so that the degree zero part of $\bar{F}_1$ is an isomorphism. The module structure implies then that in general on the elements of degree $4k$ the map $\bar{F}_1$ is the product of the top homogeneous part by a fixed power series in $V$ with leading coefficient $1$.
\end{example}

\begin{example}\label{S237}We discuss the case of $-1, 0$ and $\infty$-surgery on the right trefoil. Again, we know all the groups involved in the triangle as $(-1/n)$-surgery on the trefoil is the Seifert fibered space $\Sigma(2,3,6n+1)$, and we take a more algebraic approach. \begin{center}
\begin{tikzpicture}
\matrix (m) [matrix of math nodes,row sep=2em,column sep=1em,minimum width=2em]
  {
  \HSt_{\bullet}(\Sigma(2,3,7)) && \HSt_{\bullet}(Y_0)\\
  &\HSt_{\bullet}(S^3) &\\};
  \path[-stealth]
  (m-1-1) edge node [above]{$\check{F}_{-1}$} (m-1-3)
  (m-2-2) edge node [left]{$\check{F}_{\infty}$} (m-1-1)
  (m-1-3) edge node [right]{$\check{F}_0$} (m-2-2)  
  ;
\end{tikzpicture}
\end{center}
Again $\check{F}_{-1}$ and $\check{F}_{0}$ have degree respectively $-1$ and $0$. From this the maps are easily determined (using again the fact that the reduced Floer groups are trivial). In particular the map
\begin{equation*}
\check{F}_{-1}: \HSt_{-1}(\Sigma(2,3,7))\rightarrow \HSt_{-2}(Y_0)
\end{equation*}
is an isomorphism for degree reasons. The module structure implies then that $\check{F}_{-1}$ in an isomorphism onto the image in degrees $4k$ and $4k-1$ for $k\geq 0$, and $\check{F}_0$ is an isomorphism onto the image in degrees $4k$ and $4k+1$ for $k\geq0$. From this, as in the previous example we see that $\check{F}_{\infty}$ has top degree zero and this is an isomorphism onto the image in degree $4k+2$, and zero otherwise.
\end{example}

The computation of the third map (the one corresponding to the non spin cobordism) in the two examples we have just discussed ($\bar{F}_1$ and $\bar{F}_{\infty}$ respectively) only relies on the reducible solutions, so they hold in general for the map $\bar{F}_{1/(n+1)}$ in the $0, 1/(n+1),1/n$ surgery triangle for a knot in a homology sphere. It is important to remark that there is a difference in the case $n$ is even or odd related to relative grading of the reducibles. As we have discussed both parities in our examples, we have proved the following result concerning the image of $i_*$. 
\begin{lemma}\label{arfmap}
Suppose the knot $K$ has Arf invariant $1$. Then in the setting as above, the map $\bar{F}_{1/(n+1)}$ corresponding to the is injective of the top tower and zero of the other two towers.
\end{lemma}
The condition on the Arf invariant implies that the Rokhlin invariant of $1/(n+1),1/n$ are different, so that we are dealing with the cases of Example \ref{poincare} and \ref{S237}.
\\
\par
Finally, we know show how to use the knowledge of the third map in order to provide a previously inaccessible computation. The same ideas will be used in the next section to compute the correction terms of $(\pm1)$-surgery on alternating knots. 

\begin{proof}[Proof of Theorem \ref{figure8}]
As in \cite{OSd} have that $(+1)$-surgery on the figure eight knot is $\Sigma(2,3,7)$. Furthermore, we know that for $\spin\neq \spin_0$ the Floer groups of $E_0$ vanish because of the adjunction inequality (see Corollary $40.1.2$ in \cite{KM}). Using the fact that the reduced groups of $S^3$ and $\Sigma(2,3,7)$ are zero we can determine the Floer group $\HSt_{\bullet}(E_0,\spin_0)$ as follows. In the triangle
 \begin{center}
\begin{tikzpicture}
\matrix (m) [matrix of math nodes,row sep=2em,column sep=1em,minimum width=2em]
  {
  \HSt_{\bullet}(S^3) && \HSt_{\bullet}(E_0)\\
  &\HSt_{\bullet}(\Sigma(2,3,7)) &\\};
  \path[-stealth]
  (m-1-1) edge node [above]{$\check{F}_{\infty}$} (m-1-3)
  (m-2-2) edge node [left]{$\check{F}_{1}$} (m-1-1)
  (m-1-3) edge node [right]{$\check{F}_0$} (m-2-2)  
  ;
\end{tikzpicture}
\end{center}
the (top degree part) of the map $\check{F}_{\infty}$ is determined in light of Lemma \ref{arfmap}. The result then follows as we know the degrees of $\check{F}_{1}$ and $\check{F}_{0}$. Finally, the other cases follow from Proposition \ref{gysin}.
\end{proof}

\vspace{1cm}

\section{Surgery on alternating knots}\label{alternating}
In this final section we show how to compute the Manolescu's correction terms of the homology spheres obtained by surgery on alternating knots. This relies on the computation of the usual monopole Floer homology groups provided in \cite{OSalt} (via the isomorphism between the theories due to Kutluhan-Lee-Taubes \cite{HFHM1}, \cite{HFHM2}, \cite{HFHM3}, \cite{HFHM4}, \cite{HFHM5} and Colin-Ghiggini-Honda \cite{CGH}, \cite{CGH1}, \cite{CGH2}, \cite{CGH3} plus some additional considerations regarding absolute gradings) and some algebraic observations. We recall the main result from \cite{OSalt}. Given a knot $K$, its torsion coefficient $t_s(K)$ for an integer $s$ is defined to be
\begin{equation*}
t_s(K)=\sum_{j=1}^{\infty}j a_{|s|+j},
\end{equation*}
where the $a_s$ are the coefficients of the symmetrized Alexander polynomial of $K$. For $\sigma\in 2\mathbb{Z}$ and an integer $s$ we define
\begin{equation*}
\delta(\sigma, s)=\max\left(0, \left\lceil \frac{|\sigma|-2|s|}{4}\right\rceil\right).
\end{equation*}

\begin{thm}[Theorem $1.4$ in \cite{OSalt}]
Let $K$ be an alternating knot oriented so that $\sigma=\sigma(K)\leq 0$, and let $S^3_0(K)$ be the three manifold obtained by zero surgery. Then, letting
\begin{equation*}
b_s=(-1)^{s+\frac{\sigma}{2}}(\delta(\sigma,s)-t_s(K))
\end{equation*}
we have that:
\begin{itemize}
\item for all $s>0$ we have a $\ztwo[[U]]$ module isomorphism
\begin{equation*}
\HMt_{\bullet}(S^3_0(K), \spin_s)\cong \ztwo^{b_s}\oplus (\ztwo[U]/U^{\delta(\sigma,s)})
\end{equation*}
with the first summand supported in degree $s+\sigma/2\mod 2$ while the second summand lies in odd degree;
\item we have the isomorphism of graded modules
\begin{equation*}
\HMt_{\bullet}(S^3_0(K),\spin_0)\cong \mathcal{T}^+_{-1}\oplus \mathcal{T}^+_{-2\delta(\sigma,0)}\oplus \ztwo^{b_0}\langle \sigma/2-1\rangle.
\end{equation*}
\end{itemize}
\end{thm}

As briefly mentioned above it is important to notice that the isomorphism between monopole Floer homology and Heegaard Floer homology is only known to hold at the level of relatively graded groups. Nevertheless, in our simple case it can be seen to hold at the level absolutely graded groups thanks to the usual surgery exact triangle. Indeed the maps in the triangle
\begin{align*}
\HMt_{\bullet}(S^3) & \rightarrow \HMt_{\bullet}(S^3_0(K),\spin_0)\\
\HMt_{\bullet}(S^3_0(K),\spin_0) & \rightarrow \HMt_{\bullet}(S^3_1(K))
\end{align*}
have absolute degrees respectively $-1$ and $0$ (recall that the absolute gradings in monopole Floer homology are shifted of $-b_1(Y)/2$ with respect to those in Heegaard Floer homology), and the rank of the groups involved implies that the first map is an isomorphism onto the $\mathcal{T}^+_{-1}$ summand.
\par
The last group in the statement of the result is the sum of two particularly simple modules, namely $\T^+_{-1}$ and $\T^+_{-2k}\oplus\ztwo^{b_0}$, where the degree of the third summand is either $-2k$ or $-2k-1$. 
\\ 
\par
We are now ready to provide the main computation in the present paper.

\begin{proof}[Proof of Theorem \ref{altcorrection}]
The result follows from an application of the surgery exact triangle together with the result on simple monopole Floer homology groups discussed above. Notice that we are only interested in the image of the map $i_*$ (as a graded $\Rin$-module) as our goal is to compute Manolescu's correction terms. In particular, because of the module structure the map in the usual monopole Floer homology surgery exact triangle
\begin{equation*}
\HMt_{\bullet}(S^3)\rightarrow \HMt_{\bullet}(S^3_0(K)),
\end{equation*}
is an isomorphism onto the summand $\T^+_{-1}$ in the direct summand $\HMt_{\bullet}(S^3_0(K),\spin_0)$, while is zero onto the others. Hence $\HMt_{\bullet}(S^3_1(K))$ is isomorphic to a direct sum
\begin{equation*}
\T^+_{-2\delta(\sigma,0)}\oplus \ztwo^{b_0}\oplus(\bigoplus_{c_1(\spin)\neq 0} \HMt_{\bullet}(S^3_0(K), \spin)).
\end{equation*}
The first two summands is a module of the form of Proposition \ref{gysin}. Furthermore the Gysin sequence for the remaining summand is clear as the spin$^c$ structure are conjugate in pairs and the sequence is functorial under cobordism maps. We can then apply Proposition \ref{gysin} and determine $\HSt_{\bullet}(S^3_1(K))$, and then reconstruct $i_*\left(\HSb_{\bullet}(S^3_0(K),\spin_0)\right)$ by applying the surgery exact triangle. The details in the case of a knot with Arf invariant $1$ are analogous to those of the surgeries on the trefoil and figure eight knot discussed in Section \ref{computations}, while the case of a knot with Arf invariant $0$ is significantly simpler (as the third map in the bar version vanishes in that case). From the knowledge of $i_*\left(\HSb_{\bullet}(S^3_0(K),\spin_0)\right)$ we can then easily compute Manolescu's correction terms by iteratively applying the surgery exact triangle. As an example, we focus on the case in which $K$ has signature $-8$ and Arf invariant $1$ as it is particularly interesting in light of Remark \ref{notSFS} in the Introduction. In this case we have that
\begin{equation*}
i_*\left(\HSb_{\bullet}(S^3_1(K))\right)\cong \mathcal{S}^+_{-1,-3,-3}.
\end{equation*}
This implies (as in the proof of Theorem \ref{figure8} in the previous section) that as an absolutely graded $\Rin$-module we have
\begin{equation*}
i_*\left(\HSb_{\bullet}(S^3_0(K),\spin_0)\right)\cong (\V^+_1\oplus \V^+_0)\oplus (\V^+_{-5}\oplus \V^+_{-2})
\end{equation*}
where the $Q$ action maps the first tower onto the second tower and the third tower onto the fifth tower. Applying the surgery exact triangle again, we then obtain
\begin{equation*}
i_*\left(\HSb_{\bullet}(S^3_{-1}(K))\right)\cong \mathcal{S}^+_{1,-1,-3},
\end{equation*}
from which the result follows.
\end{proof}

\vspace{0.5cm}
We conclude by giving a proof of Proposition \ref{sfs}.
\begin{proof}[Proof of Proposition \ref{sfs}]
It is shown in \cite{MOY} that for a suitable choice of orientation of $Y$, there is a choice of metric and perturbation such that all the critical points (in the blow down) have even degree. Suppose that $Y$ is oriented in the same way as \cite{MOY}. This implies that the generators of the two and zero dimensional homology of the reducible critical submanifolds are not affected by the differential in the chain complex computing $\HSt_{\bullet}(Y,\spin)$. Consider the maximum degree at which a generator of the one dimensional homology of a reducible critical submanifold is involved in a non trivial differential. As there are no irreducible critical points of odd degree, the module structure implies that the middle tower of $i_*(\HSb_{\bullet}(Y,\spin))$ stops at that level. For the same reason, the bottom tower also stops at that level, hence we have that $\alpha(Y,\spin)$ and $\beta(Y,\spin)$ coincide. 
\par
Finally, in the orientation is the opposite, the same argument applied to $-Y$ implies that $\beta(Y,\spin)$ and $\gamma(Y,\spin)$ coincide.
\end{proof}

\vspace{1cm}

\bibliographystyle{alpha}
\bibliography{biblio}

\end{document}